%% file: main.tex
\documentclass{article}
\usepackage{mathrsfs}
\usepackage[utf8]{inputenc}
\usepackage[letterpaper,margin=1.0in]{geometry}
\input{math.tex}
\usepackage{color}
\usepackage{hyperref}
\usepackage{enumerate}
\usepackage{enumitem}
\usepackage{algorithm}
\usepackage{algorithmic}
\usepackage[numbers,sort,compress]{natbib}
\usepackage{url}
\usepackage{microtype}
\usepackage{algorithm}
\usepackage{algorithmic}
\usepackage{hyperref}
\usepackage{url}
\usepackage{microtype}
\usepackage{enumitem}
\usepackage{graphicx}
\usepackage{xspace}
\usepackage{booktabs} 
\usepackage{makecell}

\title{\vskip-0.5cm Near-Optimal Decentralized Stochastic Nonconvex Optimization \\ with Heavy-Tailed Noise}
\date{}
\author{Menglian Wang \qquad\qquad Zhuanghua Liu \qquad\qquad Luo Luo}

\begin{document}

\maketitle

\begin{abstract}
This paper studies decentralized stochastic nonconvex optimization problem over row-stochastic networks.
We consider the heavy-tailed gradient noise which is empirically observed in many popular real-world applications.
Specifically, we propose a decentralized normalized stochastic gradient descent with {\sc Pull-Diag} gradient tracking, 
which achieves approximate stationary points with the optimal sample complexity and the near-optimal communication complexity.
We further follow our framework to study the setting of undirected networks, 
also achieving the nearly tight upper complexity bounds.
Moreover, we conduct empirical studies to show the practical superiority of the proposed methods.
\end{abstract}

\section{Introduction}
This paper studies the decentralized stochastic nonconvex optimization problem 
\begin{align}\label{prob:main}
\min_{\vx \in \mathbb{R}^{d}} f(\vx):=\frac{1}{n} \sum_{i=1}^{n}f_{i}(\vx)
\end{align}
over a connected network with $n$ agents, where $f_i:\BR^d\to\BR$ is the smooth local function on the $i$th agent with the form
\begin{align*}
f_{i}(\vx):=\mathbb{E}_{\vxi_{i} \sim \mathcal{D}_{i}}[F(\vx; \vxi_{i})].
\end{align*}

We focus on the scenario that the local data $\{\vxi_{i}\}_{i=1}^n$ exhibits the heterogeneous distributions $\{\fD_i\}_{i=1}^n$ and all  $n$ agents are connected by the topology of the directed graph \citep{song2024provably,yuan2021decentlam,assran2019stochastic,bottou2018optimization}. 
Specifically, we consider local stochastic gradients with the heavy-tailed noise \citep{simsekli2019tailindexanalysisstochasticgradient,zhang2020adaptive,gurbuzbalaban2025heavy}, which is more general than the bounded variance assumption.

Decentralized stochastic nonconvex optimization is widely employed for training large-scale machine learning models.
However, most of the existing analyses are based on the assumption of bounded variance  
\citep{yuan2022revisiting,koloskova2021improved,tang2018d,lu2021optimal,koloskova2020unified,lian2017can},
which has limitations in machine learning applications.
For example, the stochastic gradient in the centralized training for language model  \cite{laurent2024linear,battash2024revisiting,zhang2020adaptive,ahn2023linear}, and reinforcement learning \cite{garg2021proximal,cayci2023provably,zhu2024robust} typically exhibits heavy-tailed noise, rather than bounded variance.
Recently, the heavy-tailed gradient noise has also been observed in decentralized training \citep{gurbuzbalaban2025heavy}.
Later, \citet{yu2025decentralized} proposed normalized stochastic gradient descent with momentum and gradient tracking for decentralized nonconvex optimization with heavy-tailed noise, while their results are restricted to undirected graphs. 

In real-world applications, the bidirectional communication protocol in a directed network may be difficult to implement because of the issues of nonuniform communication powers \citep{yang2019survey} and connectivity failures \citep{yemini2022robust}.
This motivates us to study decentralized optimization associated problems with the topology of directed graphs. 
Rather than doubly-stochastic mixing matrices for the undirected graph \cite{scaman2017optimal,sun2019distributed,kovalev2021lower,lu2021optimal,yuan2022revisiting}, the communication in the directed graph is typically characterized by the column-stochastic or row-stochastic mixing matrices, leading to the efficient gossip protocols such as \pushsum \cite{nedic2017achieving,nedic2014distributed,tsianos2012push,zeng2017extrapush,xi2017dextra,xi2017add,assran2019stochastic,qureshi2020s,song2024provably,kempe2003gossip,kungurtsev2023decentralized}, \pushpull \cite{pu2020push,Xin_2018,xin2019distributed,xin2019distributed2,qu2019accelerated}, and \pulldiag \cite{mai2019distributed,xin2019frost,ghaderyan2023fast,lu2020nesterov,xi2018linear}.
Based on these techniques,
the recent works \citep{liang2025rowstochastic,liang2025understanding} establish the efficient stochastic first-order methods that achieve the optimal sample complexity for decentralized nonconvex optimization over directed graphs, while their results cannot address the heavy-tailed noise.
\begin{table}[t]
    \centering
    \caption{We compare the \emph{overall sample complexity} and the \emph{communication complexity} of existing decentralized stochastic first-order methods with our results, where we use the notations $\tilde{\mathcal{O}}(\cdot)$ and $\tilde{\Omega}(\cdot)$ to hide logarithmic factors, and $1-\beta$ denotes the (generalized) spectral gap of the mixing matrix. The expressions for the complexity of \citet{yu2025decentralized} are complicated, and a detailed discussion is deferred to Appendix~\ref{appendix:compare}.}
    \label{table:comparsion} \vskip0.2cm
    
    \resizebox{\columnwidth}{!}{%
    \begin{tabular}{ccccc}
    \hline 
    Methods  &  Sample Complexity & Communication Complexity & Networks & Noise \\ \hline\hline\addlinespace

    \makecell{MG-\pulldiag-GT \\ {\footnotesize\citep{liang2025rowstochastic}}}  
    & $\mathcal{O}\!\left(\dfrac{L\sigma^{2}\Delta}{\epsilon^{4}}+\dfrac{L\Delta}{(1-\beta)\epsilon^{2}}\right)$ 
    & $\tilde{\mathcal{O}}\!\left(\dfrac{L\sigma^{2}\Delta}{\epsilon^{4}}+\dfrac{L\Delta}{(1-\beta)\epsilon^{2}}\right)$ 
    & row-stochastic & $p=2$ \\\addlinespace

    \makecell{DeTAG \\ {\footnotesize\citep{lu2021optimal}}} 
    & $\mathcal{O}\!\left(\dfrac{L\sigma^{2}\Delta}{\epsilon^{4}}\right)$ 
    & $\tilde{\mathcal{O}}\!\left(\dfrac{L\Delta}{(1-\beta)^{\frac12}\epsilon^{2}}\right)$ 
    & doubly-stochastic & $p=2$ \\\addlinespace    

    \hline \addlinespace

    \makecell{DNSGD-PD \\ {\footnotesize(Theorem~\ref{thm:main})}}   
    & $\mathcal{O}\!\left(\dfrac{L\sigma^{\frac{p}{p-1}}\Delta}{\epsilon^{\frac{3p-2}{p-1}}}\right)$ 
    & $\tilde{\mathcal{O}}\!\left(\dfrac{L\Delta}{(1-\beta)\epsilon^{2}}\right)$
    & row-stochastic & $p\in(1,2]$ \\ \addlinespace

    \makecell{DNSGD \\ {\footnotesize(Theorem~\ref{thm:main2})}}   
    & $\mathcal{O}\!\left(\dfrac{L\sigma^{\frac{p}{p-1}}\Delta}{\epsilon^{\frac{3p-2}{p-1}}}\right)$ 
    & $\tilde{\mathcal{O}}\!\left(\dfrac{L\Delta}{(1-\beta)^{\frac12}\epsilon^{2}}\right)$ 
    & doubly-stochastic & $p\in(1,2]$ \\ \addlinespace\hline  

    \addlinespace    
    \makecell{Lower Bounds \\ {\footnotesize\cite{liang2025rowstochastic,liu2024nonconvex}}}
    & $\Omega\!\left(\dfrac{L\sigma^{\frac{p}{p-1}}\Delta}{\epsilon^{\frac{3p-2}{p-1}}}\right)$ 
    & $\tilde{\Omega}\!\left(\dfrac{L\Delta}{(1-\beta)\epsilon^{2}}\right)$
    & row-stochastic & $p\in(1,2]$ \\ \addlinespace    

    \makecell{Lower Bounds \\ {\footnotesize\cite{liu2024nonconvex,lu2021optimal,yuan2022revisiting}}}
    & $\Omega\!\left(\dfrac{L\sigma^{\frac{p}{p-1}}\Delta}{\epsilon^{\frac{3p-2}{p-1}}}\right)$ 
    & $\Omega\!\left(\dfrac{L\Delta}{(1-\beta)^{\frac12}\epsilon^{2}}\right)$
    & doubly-stochastic & $p\in(1,2]$ \\ \addlinespace\hline
    \end{tabular}}
\end{table}

In this paper, we study the decentralized stochastic nonconvex optimization with heavy-tailed noise, i.e., the noise of the local stochastic gradient satisfies the $p$-th bounded central moment ($p$-BCM) with $p\in(1,2]$. 
Specifically, we  propose a normalized stochastic gradient
descent with \pulldiag gradient tracking for finding 
the $\epsilon$-stationary points over the row-stochastic network.
Our method successfully captures the heavy-tailed local gradient noise 
to \emph{achieve the linear speedup with respect to the number of agents~$n$}, while the previous results of linear speedup only work for bounded variance setting, i.e., the special case of $p$-BCM with $p=2$ \cite{lu2021optimal,yuan2022revisiting,liang2025rowstochastic}.   
Concretely, we provide a theoretical analysis to show that the proposed method requires the overall sample complexity of~$\mathcal{O}(L\sigma^{\frac{p}{p-1}}\epsilon^{-\frac{3p-2}{p-1}}\Delta)$
and the communication complexity of
$\tilde\fO(L(1-\beta)^{-1}\epsilon^{-2}\Delta)$,
where~$L$ is the smoothness parameter of local functions, 
$\sigma$ is the noise level, 
$\Delta$ is the initial function value gap, 
and~$1-\beta$ is the generalized spectral gap of the mixing matrix for the row-stochastic network.
It is worth noting that \emph{both our sample complexity and communication complexity (nearly) match the corresponding lower bounds}  \cite{liang2025rowstochastic,liu2024nonconvex}.
For the specific bounded variance setting (i.e., $p=2$), 
the communication complexity of our method is even sharper than the result achieved by the best-known result \cite{liang2025rowstochastic,liang2025understanding}.
We additionally follow the framework of our algorithm and analysis to address the problem over the undirected network, also achieving the optimal sample complexity and the near-optimal communication complexity.

The comparison on the complexity bounds of our methods and existing results are summarized in Table~\ref{table:comparsion}.
We further conduct experiments on training the language models to demonstrate the superiority of proposed methods.

\section{Related Work}
Decentralized stochastic optimization has been widely used in training the large-scale machine learning models, which allows the agent only communicates with its neighbors during the iterations \citep{nedic2018network}.
Most of existing works focus on the assumption of bounded variance for local stochastic gradients \citep{yuan2022revisiting,lu2021optimal,lian2017can,koloskova2021improved,koloskova2020unified,tang2018d}.
For example, \citet{lu2021optimal} and \citet{yuan2022revisiting} combined 
the standard gradient tracking 
\citep{nedic2017achieving,qu2017harnessing,nedic2009distributed} 
with Chebyshev acceleration \citep{liu2011accelerated,arioli2014chebyshev} 
to achieve nearly tight upper bounds on the sample complexity and the communication complexity for undirected networks.
More recently, \citet{liang2025understanding,liang2025rowstochastic} introduced the extended gradient tracking technique such as PUSH-SUM \citep{nedic2017achieving,nedic2014distributed,tsianos2012push,zeng2017extrapush,xi2017dextra,xi2017add,assran2019stochastic,qureshi2020s,song2024provably,kempe2003gossip,kungurtsev2023decentralized}, and PULL-DIAG \cite{mai2019distributed,xin2019frost,ghaderyan2023fast,lu2020nesterov,xi2018linear} to attain the optimal sample complexity for decentralized stochastic nonconvex first-order optimization over directed networks, 
while  the communication complexity of these methods are not tight.

The heavy-tailed gradient noise is empirically observed in many popular machine learning applications \cite{laurent2024linear,battash2024revisiting,zhang2020adaptive,ahn2023linear,garg2021proximal,cayci2023provably,zhu2024robust}, also in the decentralized training \citep{gurbuzbalaban2025heavy}.
The algorithms and theory based on gradient clipping and normalization have been extensively studied to address the heavy-tailed noise in the centralized setting \citep{gorbunov2020stochastic,gorbunov2022clipped,sadiev2023high,gorbunov2024high,zhang2020adaptive,nguyen2023high,cutkosky2020momentum,liu2023stochastic,hubler2024gradient,liunonconvex2025,sun2024gradient,he2025complexity,liu2025stochastic}.
Recently, normalized gradient methods have been extended to decentralized optimization over the undirected network \cite{li2024problem,yu2025decentralized,luo2025decentralized}.
Notably, \citet{yu2025decentralized} proposed a momentum based normalized stochastic gradient with gradient tracking, which establishes the first convergence guarantee for decentralized optimization with heavy-tailed gradient noise, while the claimed optimality is only with respect to the target accuracy in the sample complexity.

\section{Preliminaries}\label{sec:pre}

In this section, we introduce the notations and formalize the problem setting throughout this paper.

\subsection{Notations}
We use bold lowercase letters with subscripts to denote local variables, e.g., the column vector $\vx_i\in \mathbb{R}^{d}$ presents the local variable on the $i$th agent and $\nabla f_i(\vx_i)\in \mathbb{R}^{d}$ is the corresponding local gradient. 
We also use bold uppercase letters to denote matrices that aggregates local variables or local gradients from all $n$ agents, e.g.,
\begin{align*}
\mX=\begin{bmatrix}
    \vx_{1}^\top \\ \vdots \\ \vx_{n}^\top
\end{bmatrix}
\qquad\text{and}\qquad
\nabla \mF(\mX)
=\begin{bmatrix}
    \nabla f_{1}(\vx_{1})^\top \\ \vdots \\ \nabla f_{n}(\vx_{n})^\top
\end{bmatrix}    
\in \mathbb{R}^{n \times d}.
\end{align*}

For the ease of presentation, we allow the input of the function and its gradient to be organized as either a column vector or a row vector. 

Furthermore, we use $\Norm{\cdot}$ to represent the Euclidean norm of vector and the Frobenius norm of the given matrix, and use $\|\cdot\|_2$ to represent the spectral norm of the given matrix. 
The notation $\operatorname{diag}(\cdot)$ denotes a diagonal matrix whose diagonal entries are given by the specified vector.
Similarly, the notation $\operatorname{Diag}(\cdot)$ denotes the diagonal matrix whose diagonal entries are given by the specified matrix of a square matrix.

For a given square matrix $\mA\in\BR^{n\times n}$,
the notation $\mA^m$ represents $\mA$ to the power of $m$.
Additionally, we use the notation $\mI\in\BR^{n\times n}$ to
present the identity matrix
and denote the all-ones vector as $\mathbf{1}=[1, \ldots, 1]^{\top} \in \mathbb{R}^{n}$. 

\subsection{Problem Setting}
We focus on the problem over the directed and strongly-connected network $\fG=(\fV,\fE)$, where the node set $\fV=[n]$ contains all $n$ agents 
and the edge set $\fE=\{(i\!\rightarrow\!j):\text{there exists a path leads from agent $i$ to agent $j$}\}$ describes the connectivity of the network.
The communication step is typically characterized by multiplying the aggregated variables by the mixing matrix.
Specifically, we suppose the mixing matrix associate with the network satisfies the following standard assumption \cite{mai2019distributed,xin2019frost,ghaderyan2023fast,lu2020nesterov,xi2018linear,liang2025rowstochastic}.

\begin{asm}\label{asm:primitive}
    The mixing matrix $\mA=[a_{ij}]\in\BR^{n\times n}$ associated with the network $\fG=(\fV,\fE)$ satisfies  $\mA \mathbf{1}=\mathbf{1}$ and the entries of $\mA$ holds $a_{ij} \in (0,1]$ if $(i \rightarrow j) \in \mathcal{E}$ or~$i=j$, and~$a_{ij}=0$ otherwise for all $i,j\in[n]$.
\end{asm}

For a given aggregated variable $\mX=[\vx_1^\top;\dots;\vx_n^\top]\in\BR^{n\times d}$, we describe $K$ communication steps associated with the mixing matrix $\mA\in\BR^{n\times n}$ by the initialization $\mZ_0=\mX$ and the recursion
\begin{align}\label{eq:Zk}
    \mZ_{k}=\mA\mZ_{k-1}= \dots =\mA^k\mZ_0=\mA^k\mX.
\end{align}

The row-stochastic matrix $\mA\in\BR^{n\times n}$ has the following properties \citep{oskar1907Zur}.

\begin{prop}[{\citet{oskar1907Zur}}]\label{prop:Perron-Frobenius}
Under Assumption \ref{asm:primitive}, there exists a unique equilibrium vector 
$\vpi\in \mathbb{R}^n$ with positive entries such that
\begin{align*}
   \vpi^{\top} \mA=\vpi^{\top}, \quad \vpi^{\top} \mathbf{1}=1,\quad \text{and} \quad \lim_{k \rightarrow +\infty} \mA^{k}=\mathbf{1}\vpi^{\top}. 
\end{align*}
\end{prop}

For the equilibrium vector $\vpi=[\pi_1,\dots,\pi_n]^\top \in \mathbb{R}^n$ in Proposition \ref{prop:Perron-Frobenius}, 
we define  the matrix
${\bf\Pi}:= \operatorname{diag}(\vpi)\in\BR^{n\times n}$
and the weighted matrix norm 
$\Norm{\mA}_{\vpi} := \|{\bf\Pi}^{1/2}\mA{\bf\Pi}^{-1/2}\|_2$.
We also use $1-\beta$ to present the generalized spectral gap of the row-stochastic matrix $\mA$, where
\begin{align*}
    \beta:=\|\mA-\mathbf{1 } \vpi^{\top}\|_{\vpi}. 
\end{align*}
Additionally, we denote the equilibrium skewness of the row-stochastic matrix $\mA$ as
\begin{align*}
    \kappa:=\frac{\max_{i\in[n]} \pi_i}{\min_{i\in[n]} \pi_i} \in[1,+\infty).
\end{align*}

\begin{remark}
In the view of equation (\ref{eq:Zk}),
the parameter $\beta$ characterizes the convergence rate of~$\mZ_k$ to the weighted average $\vone_n\vpi^\top\mX$.
On the other hand, the equilibrium skewness $\kappa$ captures the disagreement between the equilibrium vector $\pi$ and the uniform vector $\vone_n/n$.
Intuitively, the sequence~$\{\mZ_k\}$ convergence to $\vone_n\vpi^\top\mX$ rapidly for small $\beta$ 
and $\vpi$ is close to $\vone_n/n$ for small $\kappa$.
\end{remark}

We impose the following assumption on problem (\ref{prob:main}).
\begin{asm}\label{asm:smoothness}
Suppose each $f_i$ is $L$-smooth, i.e., there exists $L>0$ such that 
$\Norm{\nabla f_{i}(\vx)-\nabla f_{i}(\vy)} \leq L\Norm{\vx-\vy}$
for all $i\in [n]$ and $\vx, \vy \in \mathbb{R}^{d}$. 
Additionally, the objective $f$ is lower bounded by
$f^*=\inf_{\vx\in\BR^d}f(\vx)>-\infty$.
\end{asm}

We focus on the distributed stochastic first-order method, which can access the local stochastic first-order oracle satisfying the $p$-BCM assumption as follows.

\begin{asm}\label{asm:pBCM}
For given $i\in[n]$ and $\vx\in\BR^d$, the $i$th agent can draw $\vxi_i\sim\fD_i$ and access the local stochastic first-order oracle $\nabla F_i(\vx;\vxi_i)$ which satisfies
$\BE[\nabla F_i(\vx;\vxi_i)] = \nabla f_i(\vx)$ and
$\BE\left[\Norm{\nabla F_i(\vx;\vxi_i) - \nabla f_i(\vx)}^p\right] \leq \sigma^p$
for some $\sigma>0$ and $p\in(1,2]$. 
\end{asm}

\section{The Algorithm and Main Results}\label{sec:main}
We propose decentralized normalized stochastic gradient descent with \pulldiag gradient tracking (DNSGD-PD) in Algorithm~\ref{alg:PDNSGD}, which integrates the techniques of stochastic gradient normalization, \pulldiag gradient tracking, 
and protocol of multi-consensus gossip to handle the directed communication and heavy-tailed local gradients noise. 
Note that it guarantees the inverse of $\mD_{t}=\operatorname{Diag}(\mA^{(t+1)K})$ be well-defined by an appropriate setting of $K$, since the diagonal entries for the power of $\mA$ can be bounded as follows.

\begin{prop}\label{lem:Bounded Diagonals} 
Under Assumption \ref{asm:primitive}, then the diagonal entries of $\mA^K$ holds
\begin{align*}
       [\mA^{K}]_{i i}>0 \qquad \text{and} \qquad [\mA^{K}]_{i i}^{-1} \leq \theta:=2n\kappa
\end{align*}
for all $K \geq 3(1+\ln (\kappa n))/{(1-\beta)}$ and $i \in[n]$.
\end{prop}

\begin{algorithm*}[h]
\caption{Normalized Stochastic Gradient Descent with \pulldiag Gradient Tracking (DNSGD-PD)}
\label{alg:PDNSGD}
\begin{algorithmic}[1]
\STATE \textbf{Input:} $\bx_0\in\BR^{1\times d}$, $\mA\in\BR^{n\times n}$, $\eta>0$, $\hat K,K,T,b\in\BN$. \\[0.12cm]
        \STATE $\mX_0=\vone\bx_0$;\quad independently sample $\vxi_{0,i,k}\sim\fD_i$ for all $k\in[b]$ at each agent $i$;  \\[0.12cm]
        \STATE \label{line:G0} $\mG_0 = \begin{bmatrix}
              {\displaystyle\frac{1}{b}\sum_{k=1}^{b} \nabla F_1(\bx_0;\vxi_{0,1,k})} & \dots & 
              {\displaystyle\frac{1}{b}\sum_{k=1}^{b} \nabla F_n(\bx_0;\vxi_{0,n,k})}
            \end{bmatrix}^\top$;\quad $\mV_0 = \mA^{\hat K}\mD_0^{-1}\mG_0$; \\[0.12cm]
        \STATE \textbf{for} $t = 0, 1, \dots, {T-1}$ \textbf{do} \\[0.12cm]   		     
        \STATE\quad\label{line:U}$\mU_t = \begin{bmatrix}
            \dfrac{\vv_{t,1}}{\Norm{\vv_{t,1}}} & \dots & \dfrac{\vv_{t,n}}{\Norm{\vv_{t,n}}}
        \end{bmatrix}^\top$; \quad $\mX_{t+1} = \mA^K\left(\mX_{t} - \eta\mU_t\right)$; \label{line:update-X} \\[0.12cm]
        \STATE\quad independently sample $\vxi_{t+1,i,k}\sim\fD_i$ for all $k\in[b]$ at each agent $i$; \\[0.12cm]
        \STATE\quad \label{line:Gt} $\displaystyle{\mG_{t+1} =\begin{bmatrix}
        {\displaystyle\frac{1}{b} \sum_{k=1}^{b} \nabla F_{1}(\vx_{t+1,1}; \vxi_{t+1,1,k})} 
        & \cdots &
        {\displaystyle\frac{1}{b} \sum_{k=1}^{b} \nabla F_{n}(\vx_{t+1,n}; \vxi_{t+1,n,k})}\end{bmatrix}^\top}$; \\[0.15cm]
        \STATE\quad \label{line:Vt1} $\mD_t= \operatorname{Diag} (\mA^{(t+1)K})$;\quad $\mD_{t+1}= \operatorname{Diag} (\mA^{(t+2)K})$; \quad $\mV_{t+1} = \mA^{K}(\mV_{t}
        + \mD_{t+1}^{-1} \mG_{t+1}- \mD_{t}^{-1} \mG_{t})$; \\[0.2cm]       
        \STATE\textbf{end for} \\[0.12cm]
        \STATE $\hat\vx_i\sim{\rm Unif}\{\vx_{0,i},\dots,\vx_{T-1,i}\}$ for each agent $i$. \\[0.12cm] 
		\STATE \textbf{Output:} $\hat\vx_i$ for each agent $i$.
\end{algorithmic}
\end{algorithm*}

We now provide the convergence analysis for our method.
Motivated by Proposition \ref{prop:Perron-Frobenius}, we focus on the function value at the weighted average vector
\begin{align*}
    \vw_t := \vpi^\top \mX_t = \sum_{i=1}^n \pi_{i}\vx_{t,i}
\end{align*}
to establish the following descent lemma.
\begin{lem}\label{lem:descent-mean}
Under Assumptions \ref{asm:primitive}--\ref{asm:pBCM}, Algorithm \ref{alg:PDNSGD} holds
\begin{align}\label{eq:descent-f}
\begin{split}
    & \BE\left[ f(\vw_{t+1})\right]\\
\leq &f(\vw_{t})-\eta \Norm{\nabla f(\vw_{t})} + \frac{\eta^{2} L}{2}  + \BE\left[\frac{2\eta}{n} \Norm{n\nabla f(\vw_{t}) -\vpi^{\top} \mV_{t}} +\frac{2\eta}{n}  \Norm{\mV_{t}-\mathbf{1} \vpi^{\top} \mV_{t}}\right].
\end{split}  
\end{align}
\end{lem}

Compared with the analysis of normalized gradient descent methods on a single machine \cite{hubler2024gradient}, the second line of equation (\ref{eq:descent-f}) includes the additional terms of~$n^{-1}\Norm{n\nabla f(\vw_{t}) - \vpi^{\top} \mV_{t}}$ 
and $\Norm{\mV_{t}-\mathbf{1} \vpi^{\top} \mV_{t}}$,
where the first one quantifies the difference between the global gradient and the \pulldiag descent direction, and the second one is the consensus error. 

The key to achieving \emph{the linear speedup} of proposed method is the following upper bound for the expectation of the term $n^{-1}\Norm{n\nabla f(\vw_{t}) - \vpi^{\top} \mV_{t}}$.

\begin{lem}\label{lem:Estimate descent deviation} Under Assumptions \ref{asm:primitive}--\ref{asm:pBCM}, Algorithm \ref{alg:PDNSGD} holds
 \begin{align}\label{eq:derivation}
 \begin{split}     
n^{-1}\BE\left[\Norm{n\nabla f(\vw_{t})-\vpi^{\top} \mV_t}\right] \leq & \frac{\left(1+\sqrt{\kappa}\theta \beta^{(t+1)K}\right)L}{\sqrt{n}} \Norm{\mX_{t}-\mathbf{1} \vpi^{\top} \mX_{t}}  + \theta  \sqrt{\kappa} \beta^{(t+1)K}\Norm{\nabla f(\vw_{t})} \\
& + \frac{2\sqrt{2\kappa}\theta\beta^{(t+1)K}\sigma}{b^{1-1 / p}}+\frac{8\sigma}{(nb)^{1-1 / p}}.
 \end{split}
\end{align}
\end{lem}

Note that taking appropriate $K$ can guarantee $\beta^{(t+1)K}$ be sufficient small, which implies the term~$8\sigma/((nb)^{1-1 / p})$ dominates the upper bound (\ref{eq:derivation}).
Hence, each agent can take the minibatch size of $b=\Theta\big(n^{-1}\sigma^{\frac{p}{p-1}}\epsilon^{-\frac{p}{p-1}}\big)$ to achieve the desired accuracy $\epsilon$, leading to the linear speedup with respect to the agent number $n$.

For the consensus error, we introduce the notation 
\begin{align*}
\rho:=\Norm{\mA^{K}-\mathbf{1} \vpi^{\top}} \leq \sqrt{n\kappa}\beta^K
\end{align*}
to characterize the convergence of the multi-step gossip,
where the derivation for the upper bound $\sqrt{n\kappa}\beta^K$ can be founded in Lemma \ref{lem: matrixerror} of Appendix \ref{sec:supporting-lemmas}.
Then we establish the recursions for the consensus errors in following lemmas.

\begin{lem}\label{lem:consensus_x}
Under Assumptions \ref{asm:primitive}--\ref{asm:pBCM}, Algorithm \ref{alg:PDNSGD} holds
\begin{align*}
 \Norm{\mX_{t+1}-\mathbf{1}\vpi^\top  \mX_{t+1}} \leq \rho(\Norm{\mX_{t}-\mathbf{1}\vpi^\top  \mX_{t}} + 2\sqrt{n} \eta).
\end{align*}
\end{lem}

\begin{lem} \label{lem:consensus_v}
Under Assumptions \ref{asm:primitive}--\ref{asm:pBCM}, Algorithm \ref{alg:PDNSGD} holds
\begin{align*}
&\BE\left[\Norm{\mV_{t+1}-\mathbf{1}\vpi^{\top} \mV_{t+1}}\right]\\
\leq&\rho\Big(\Norm{\mV_{t}-\mathbf{1}\vpi^{\top} \mV_{t}}+\theta L(1+\rho+2 n^2\kappa^{3/2} \beta^{(t+1)K})\Norm{\mX_{t}-\mathbf{1} \vpi^{\top} \mX_{t}}  \\
&+2 \theta n^{5/2}\kappa^{3/2} \beta^{(t+1)K}\Norm{\nabla f(\vw_{t})} +  (2 \rho +1)\theta L \eta \sqrt{n}+ \frac{4 \sqrt{2 n}\theta\sigma(1+n^2\kappa^{3/2} \beta^{(t+1)K})}{b^{1-1 / p}}\,\Big).
\end{align*}
\end{lem}

The upper bounds provided by Lemmas \ref{lem:descent-mean}--\ref{lem:consensus_v} motivate us to construct the following Lyapunov function
\begin{align*}
\small\begin{split}
\Phi_t := &f(\vw_{t})  +  \frac{32\eta L}{\sqrt{n}}\Norm{\mX_{t}-\mathbf{1}\vpi^\top  \mX_{t}} + \frac{4\eta}{n} \Norm{\mV_{t}-\mathbf{1}\vpi^{\top} \mV_{t}}.
\end{split}
\end{align*}

Applying Lemmas \ref{lem:descent-mean}--\ref{lem:consensus_v} with leads to
\begin{align*}
\begin{split}
\BE\left[\frac{1}{T}\sum_{t=0}^{T-1}\Norm{\nabla f(\vw_t)}\right] 
\leq \frac{2\BE[\Phi_0-\Phi_T]}{T\eta}  + \fO(\epsilon).
\end{split}
\end{align*}
By appropriate settings of $\hat K$ and $K$, the above result implies each agent can attain an $\epsilon$-stationary point in expectation

\begin{thm}\label{thm:main}
Under Assumptions \ref{asm:primitive}--\ref{asm:pBCM}, running Algorithm~\ref{alg:PDNSGD} with parameter settings
$T=\Theta\left(L\Delta\epsilon^{-2}\right)$, $\eta=\Theta(\epsilon/L)$, 
$b=\Theta\big(n^{-1} \sigma^{\frac{p}{p-1}}\epsilon^{-\frac{p}{p-1}}\big)$, 
and $\hat K = K = \tilde{{\Theta}} \left(1/(1-\beta)\right)$
guarantees
$\BE[\Norm{\nabla f({\hat\vx}_i)}] \leq \epsilon$
for all $i\in[n]$, where $\Delta:=f(\bar\vx_0)-f^*$.
Additionally, it takes the overall sample complexity of
$\mathcal{O}(L \sigma^{\frac{p}{p-1}}\epsilon^{-\frac{3p-2}{p-1}}\Delta)$ and the communication complexity of
${\tilde\fO}(L(1-\beta)^{-1}\epsilon^{-2}\Delta)$.
\end{thm}

\section{The Optimality and Discussion}
In fact, Theorem \ref{thm:main} indicates Algorithm \ref{alg:PDNSGD} has the optimal complexity on both computation and communication.
We start our discussion with the following proposition.

\begin{prop}[{\citet[Theorem 3.3]{liu2024nonconvex}}]
\label{prop:main-lower-bound}
For any given $L,\sigma,\Delta>0$, $p\in\left(1,2\right]$, and sufficient small $\epsilon>0$,
there exist a function $h:\BR^d\to\BR$ which is $L$-smooth and lower bounded by 
$h(\vx_0)-\inf_{\vx\in\BR^d}h(\vx)\leq\Delta$, 
and a stochastic first-order oracle that can draw $\vxi\sim\fD$ and return $H(\vx;\vxi)$ satisfying $\BE[\nabla H(\vx;\vxi)] = \nabla h(\vx)$ and $\BE\left[\Norm{\nabla H(\vx;\vxi) - \nabla h(\vx)}^p\right] \leq \sigma^p$ for given $\vx\in\BR^d$.
In order to find an $\epsilon$-stationary point $\hat\vx$, i.e., $\BE\left[\left\Vert \nabla h(\hat\vx)\right\Vert \right]\leq\epsilon$,
any stochastic first-order oracle algorithm starting from the point $\vx_0\in\BR^d$ requires at least $\Omega(L\sigma^{\frac{p}{p-1}}\epsilon^{-\frac{3p-2}{p-1}}\Delta)$ stochastic first-order oracle calls.
\end{prop}

According to Proposition \ref{prop:main-lower-bound}, we can construct a hard instance for our decentralized setting by taking $f_i=h$, and $F_i=H$ for all $i\in[n]$, then the lower bound of $\Omega(L\sigma^{\frac{p}{p-1}}\epsilon^{-\frac{3p-2}{p-1}}\Delta)$ on the stochastic first-order oracle complexity also applies to decentralized stochastic nonconvex optimization under the $p$-BCM assumption. This implies Algorithm \ref{alg:PDNSGD} has the optimal sample complexity.

For the communication complexity, we consider following result with respect to the noiseless local gradient oracle.

\begin{prop}[{\citet[Appendix A.2]{liang2025rowstochastic}}]\label{prop:main-lower-bound2}
For any given $L,\Delta>0$, $n\geq2$, and $\beta\in[0.01, 1-1/n]$, there exists 
a set of $L$-smooth functions $\{f_i:\BR^d\to\BR\}_{i=1}^n$ such that $f(\vx_0)-\inf_{\vx\in\BR^d}f(\vx)\leq\Delta$ with $f=\frac{1}{n}\sum_{i=1}^n f_i$ and a row-stochastic matrix $\mA\in\BR^{n\times n}$ with $\ln\kappa=\Omega(n(1-\beta))$.
In order to find an $\epsilon$-stationary point $\hat\vx$, \mbox{i.e., $\|\nabla f(\hat\vx)\|\leq\epsilon$},
any decentralized exact first-order methods over the network associated with the row-stochastic matrix $\mA$ requires at least 
$\Omega(L(1-\beta)^{-1}\epsilon^{-2}\Delta\ln\kappa)$ communication rounds. 
\end{prop}

It is worth noting that the exact local first-order oracle also satisfied our $p$-BCM condition for any $p\in(1,2]$ and $\sigma>0$ by taking Assumption~\ref{asm:pBCM} with $\nabla F_i(\vx;\vxi) = \nabla f_i(\vx)$.
Therefore, the lower bound of $\Omega(L(1-\beta)^{-1}\epsilon^{-2}\Delta\ln\kappa)$ as shown in Proposition \ref{prop:main-lower-bound2} also applies to our decentralized stochastic first-order nonconvex optimization under
the $p$-BCM assumption.
This implies Algorithm \ref{alg:PDNSGD} has the near-optimal communication complexity.

In the special case of $p=2$, a recent work \citep{liang2025rowstochastic} claims MG-\pulldiag-GT method attains the near-optimal complexity of $\tilde{\fO}(L \sigma^{2}\epsilon^{-4}\Delta+L(1-\beta)^{-1}\epsilon^{-2}\Delta)$.
However, their upper bound analysis does not distinguish between the sample complexity and the communication complexity.
In fact, both the sample complexity and the communication complexity~of MG-\pulldiag-GT \citep[Algorithm 3]{liang2025rowstochastic} for achieving an $\epsilon$-stationary point are $\tilde{\fO}(L\sigma^{2}\epsilon^{-4}\Delta+L (1-\beta)^{-1}\epsilon^{-2}\Delta)$, which are strictly worse than those in our Theorem \ref{thm:main} with $p=2$.

\section{Results for Undirected Networks}\label{sec:undirected}
This section follows the design and analysis in Section \ref{sec:main} to study the problem over undirected network $\fG=(\fV,\fE)$, where 
the node set $\fV=[n]$ contains all $n$ agents 
and the edge set $\fE=\{(i,j):\text{the agents $i$ and $j$ are connected}\}$ describes the connectivity. 
Compared with the settings in Section \ref{sec:pre}, we suppose the mixing matrix $\mA\in\BR^{n\times n}$ be doubly-stochastic, i.e.,
we strength Assumption \ref{asm:primitive} as follows \citep{ye2023multi,yuan2022revisiting,lu2021optimal}.

\begin{asm}\label{asm:primitive2}
    The mixing matrix $\mA=[a_{ij}]\in\BR^{n\times n}$ associated with the undirected network $\fG=(\fV,\fE)$ is symmetric. It satisfies $a_{ij} \in (0,1]$ if $(i,j) \in \mathcal{E}$ or~$i=j$, and~$a_{ij}=0$ otherwise. In addition, it holds $\mathbf{0} \preceq \mA \preceq \mI$, $\mA^\top \boldsymbol{1} = \mA \boldsymbol{1} = \boldsymbol{1}$, and $\operatorname{null}(\mI-\mA)=\operatorname{span}(\boldsymbol{1})$.
\end{asm}

In the view of Proposition \ref{prop:Perron-Frobenius} and Assumption \ref{asm:primitive2},
we set 
$\vpi=\frac{1}{n}\vone$
throughout this section. 
We use the standard gradient tracking to replace the matrix set. This implies the skewness $\kappa=1$, so that the diagonal matrix $\mD_t$ in Algorithm \ref{alg:PDNSGD} can be set to the identity matrix for the undirected network.

Consequently, the vector $\vw_t$ reduces to the mean vector, i.e.,
\begin{align*}
    \vw_t = \frac{1}{n}\sum_{i=1}^n \vx_{t,i}:=\bar\vx_t,
\end{align*}
and the parameter $\beta$ reduce to
$\beta= \big\|\mA - \frac{1}{n}\boldsymbol{1}\boldsymbol{1}^\top\big\|_2$.

Under Assumption \ref{asm:primitive2}, the communication described by multiplication $\mA^K$ shown in equation (\ref{eq:Zk}) can be improved by introducing Chebyshev acceleration \cite{liu2011accelerated,arioli2014chebyshev,ye2023multi} for the undirected network.
We formally present this subroutine in Algorithm~\ref{alg:fm}, which has the following property.

\begin{prop}[{\citet{ye2023multi}}]\label{prop:Chebyshev acceleration}
Under Assumption \ref{asm:primitive2}, the output of Algorithm \ref{alg:fm} holds $\frac{1}{n}\boldsymbol{1}^\top \mY_K = \bar{\mY}_0$ and
\begin{align*}
    &\Norm{\mY_K-\boldsymbol{1}\bar{\vy}_0}\leq c_1\big(1-c_2 \sqrt{1-\beta}\,\big)^K\Norm{\mY_0-\boldsymbol{1}\bar{\vy}_0}
\end{align*}
where $\bar{\vy}^{(0)}=\frac{1}{n}\boldsymbol{1}^\top \mY^{(0)}, ~~c_1=\sqrt{14}, ~~and~~c_2=1-1/\sqrt{2}$.
\end{prop}

According to Proposition \ref{prop:Chebyshev acceleration}, we define
\begin{align*}
    \tilde\rho := c_1\big(1-c_2\sqrt{1-\beta}\,\big)^K
\end{align*}
which characterize the convergence of Algorithm \ref{alg:fm}.

\begin{algorithm}[t]
	\caption{$\AG(\mZ_0,\mA, K)$} \label{alg:fm}
	\begin{algorithmic}[1]
		\STATE $\mZ_{-1}=\mZ_{0}$;~~~$\eta_z=(1-\sqrt{1-\beta^2}\,)/(1+\sqrt{1-\beta^2}\,)$ \\[0.1cm]
		\STATE \textbf{for} $k = 0, 1, \dots, K-1$ \textbf{do}\\[0.1cm]
		\STATE\quad $\mZ_{k+1}=(1+\eta_z)\mA\mZ_{k}-\eta_z\mZ_{k-1}$ \\[0.1cm]
		\STATE\textbf{end for} \\[0.1cm]
		\STATE \textbf{Output:} $\mZ_{K}$
	\end{algorithmic}
\end{algorithm}

Based on the above discussion, we adapt Algorithm \ref{alg:PDNSGD} to the undirected network by modifying the update rules of the variables $\mX_{t+1}$ and $\mV_{t+1}$ to
\begin{align}\label{eq:modify-XV}
\begin{cases}
\mX_{t+1} = \AG\left(\mX_{t} - \eta\mU_t, \mA, K\right), \\
\mV_{t+1} = \AG\left(\mV_{t} + \mG_{t+1} - \mG_{t}, \mA, K\right).    
\end{cases}
\end{align}
Similarly, we also apply Chebyshev acceleration to set 
\begin{align}\label{eq:modify-X0}
\mV_0 = \AG\big(\mG_0, \mA, \hat K\big).
\end{align}
We present the complete  procedure of Algorithm \ref{alg:PDNSGD} with above modifications in Algorithm \ref{alg:DNSGD} of Appendix \ref{appendix:undirected}.

\begin{remark}
By modifications (\ref{eq:modify-XV}) and (\ref{eq:modify-X0}),
the procedure of Algorithm \ref{alg:PDNSGD} reduces to the recent proposed decentralized normalized stochastic  gradient descent (DNSGD) \citep{luo2025decentralized}.
Note that the motivation and analysis of previous work for DNSGD \citep{luo2025decentralized} is target to solve the stochastic nonconvex problem under the relaxed smoothness and bounded variance assumptions.
Additionally, existing convergence guarantee of DNSGD requires the condition of bounded gradient dissimilarity. 
In contrast, the analysis in the remainder of this section focuses on the settings of standard smooth and heavy-tailed noise.
\end{remark}

We then provide the complexity analysis of the modified algorithm for undirected networks.
We present the descent lemma with respect to the mean vector 
$\bar\vx_t=\frac{1}{n}\sum_{i=1}^n \vx_{t,i}$.

\begin{lem}\label{lem:descent-undirected}
Under Assumptions \ref{asm:smoothness}--\ref{asm:primitive2}, running Algorithm~\ref{alg:PDNSGD} by modifying the updates of $\mX_{t+1}$ and $\mV_{t+1}$ to equations~(\ref{eq:modify-XV}) and (\ref{eq:modify-X0}) holds
\begin{align*}
    &\BE\left[ f(\bar{\vx}_{t+1})\right] \\
    \leq &f(\bar{\vx}_{t})-\eta \Norm{\nabla f(\bar{\vx}_{t})} + \frac{\eta^{2} L}{2}+ \BE\left[2\eta\Norm{\nabla f(\bar{\vx}_{t}) -\frac{1}{n}\boldsymbol{1}^\top \mV_{t}}+ \frac{2\eta}{\sqrt{n}}\Norm{\mV_{t}-\frac{1}{n}\boldsymbol{1}\boldsymbol{1}^\top  \mV_{t}}\right].
\end{align*}
\end{lem}

Note that the form of Lemma \ref{lem:descent-undirected} is different from that of Lemma \ref{lem:descent-mean}, since the term of $\BE\left[\Norm{\nabla f(\bar{\vx}_{t}) -\frac{1}{n}\boldsymbol{1}^\top \mV_{t}}\right]$ is not identical to the term of $n^{-1}\BE\left[\Norm{n\nabla f(\vw_{t}) - \vpi^{\top} \mV_{t}}\right]$ in Lemma \ref{lem:descent-mean} even by taking $\vw_t=\bar\vx_t$ and $\vpi=\frac{1}{n}\vone$.
This difference comes from the modification $\mD_t=\mI$ in the update of $\mV_{t+1}$, which is unnecessary to address the skewness for $\vpi=\frac{1}{n}\vone$.
Nevertheless, the following lemma also shows that $\BE\left[\Norm{\nabla f(\bar{\vx}_{t}) -\frac{1}{n}\boldsymbol{1}^\top \mV_{t}}\right]$ is dominated by the term of $\Theta(\sigma/((nb)^{1-1/p}))$, where the factor $nb$ indicates the linear speedup for undirected networks.

\begin{lem}\label{lem:Estimate descent deviation undirected}
    Under the setting of Lemma \ref{lem:descent-undirected}, we have
    \begin{align*}
        \BE\left[\Norm{\nabla f(\bar{\vx}_{t}) -\frac{1}{n}\boldsymbol{1}^\top \mV_{t}}\right] \leq \frac{L}{\sqrt{n}}\Norm{\mX_t-\frac{1}{n}\boldsymbol{1}\boldsymbol{1}^\top\mX_t} + \frac{8\sigma}{(nb)^{1-1/p}}
    \end{align*}
\end{lem}

We then establish the recursion for the consensus errors.

\begin{lem}\label{lem:consensus_x_undirected}
Under the setting of Lemma \ref{lem:descent-undirected}, we have
\begin{align*}\small
\begin{split}
    \Norm{\mX_{t+1}-\frac{1}{n}\mathbf{1}\boldsymbol{1}^{\top} \mX_{t+1}}\leq\tilde\rho\left(\Norm{\mX_{t}-\frac{1}{n}\mathbf{1}\boldsymbol{1}^{\top} \mX_{t}}+\sqrt{n}\eta\right).
\end{split}
\end{align*}
\end{lem}

\begin{lem}\label{lem:consensus_v_undirected}
Under the setting of Lemma \ref{lem:descent-undirected}, we have
\begin{align*}\small
\begin{split}    
&\BE\left[\Norm{\mV_{t+1}-\frac{1}{n}\mathbf{1}\boldsymbol{1}^{\top} \mV_{t+1}}\right]\\
\leq&\tilde\rho\left(\,\Norm{\mV_{t}-\frac{1}{n}\mathbf{1}\boldsymbol{1}^{\top} \mV_{t}}+(\tilde\rho+1)L\Norm{\mX_{t}-\frac{1}{n}\mathbf{1} \boldsymbol{1}^{\top} \mX_{t}} +  (2 \tilde\rho +1) L \eta \sqrt{n} + \frac{4 \sqrt{2 n}\sigma}{b^{1-1 / p}}\right).
\end{split}
\end{align*}
\end{lem}

Combing Lemmas \ref{lem:descent-undirected}--\ref{lem:consensus_v_undirected}, we achieve our main results for the setting of undirected networks.

\begin{thm}\label{thm:main2}
Under Assumptions \ref{asm:smoothness}--\ref{asm:primitive2}, running Algorithm~\ref{alg:PDNSGD} by modifying the updates of $\mX_{t+1}$ and $\mV_{t+1}$ to equations~(\ref{eq:modify-XV}) and (\ref{eq:modify-X0}) and 
by taking 
$T=\Theta\left(L\Delta\epsilon^{-2}\right)$, $\eta=\Theta(\epsilon/L)$, 
$b=\Theta\big(n^{-1} \sigma^{\frac{p}{p-1}}\epsilon^{-\frac{p}{p-1}}\big)$, 
and $\hat K = K = \tilde{{\Theta}}(1/(1-\beta)^{\frac{1}{2}})$
guarantees
$\BE[\Norm{\nabla f({\hat\vx}_i)}] \leq \epsilon$
for all $i\in[n]$, where $\Delta:=f(\bar\vx_0)-f^*$.
Additionally, it takes the overall sample complexity of
$\mathcal{O}(L \sigma^{\frac{p}{p-1}}\epsilon^{-\frac{3p-2}{p-1}}\Delta)$ and the communication complexity of
${\tilde\fO}(L(1-\beta)^{-\frac{1}{2}}\epsilon^{-2}\Delta)$.
\end{thm}

We emphasize that both the sample complexity and the communication complexity achieved by Theorem~\ref{thm:main2} are nearly tight for undirected networks.
Recall the discussion on the lower bound after Proposition \ref{prop:main-lower-bound} does not depend on the topology of the network. 
Therefore, it also implies our overall sample complexity of $\mathcal{O}(L \sigma^{\frac{p}{p-1}}\epsilon^{-\frac{3p-2}{p-1}}\Delta)$ is optimal for  undirected networks.

Additionally, \citet[Instance 2 in Appendix~D.1]{yuan2022revisiting} showed the decentralized nonconvex optimization over the undirected network has the communication lower bound of~$\Omega(L(1-\beta)^{-\frac{1}{2}}\epsilon^{-2}\Delta)$.
In fact, the construction of such lower bound depends on the noiseless local gradient oracle. Similar to the discussion after Proposition \ref{prop:main-lower-bound2}, it also works for our setting of stochastic local gradient oracle. 
Hence, our communication complexity of ${\tilde\fO}(L(1-\beta)^{-\frac{1}{2}}\epsilon^{-2}\Delta)$ is near-optimal for undirected networks.

Recently, \citet{yu2025decentralized} proposed normalized stochastic gradient descent with momentum and gradient
tracking (GT-NSGDm) for stochastic nonconvex optimization with heavy-tailed noise over undirected networks.
The overall sample complexity of GT-NSGDm depends on $\fO(\epsilon^{\frac{3p-2}{p-1}}\big)$ which is tight with respect to the accuracy $\epsilon$, while it is not optimal to the other parameters such as $\sigma$, $L$, and $\beta$.
Moreover, the communication complexity of GT-NSGDm also depends on~$\fO(\epsilon^{\frac{3p-2}{p-1}}\big)$, which is more expensive than ours. 
Please refer to the detailed discussion in Appendix \ref{appendix:compare}.

\begin{figure}[t]
  \centering
  \begin{tabular}{c@{\hspace{1.5cm}}c}
       \includegraphics[scale=0.17]{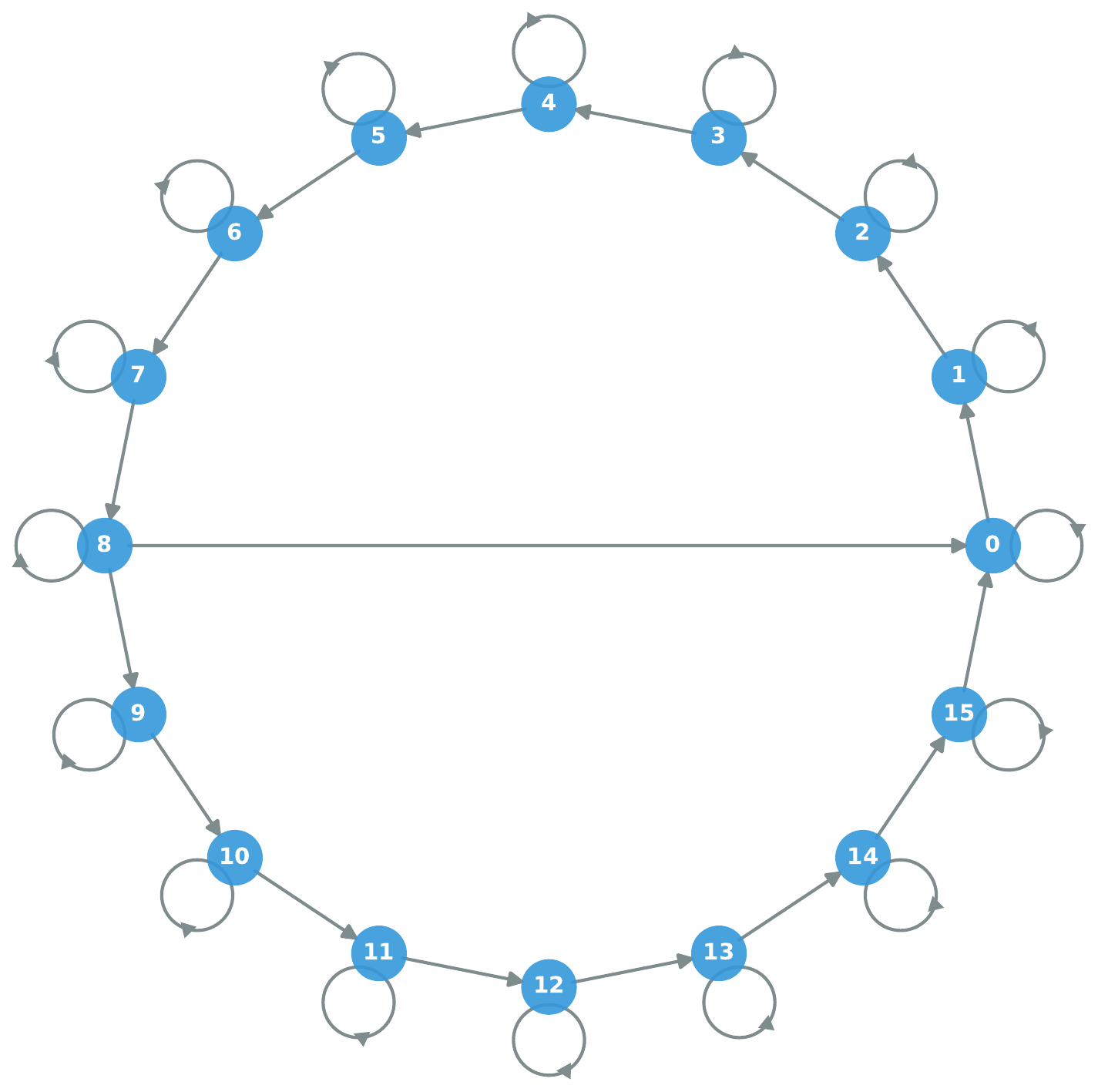} &
       \includegraphics[scale=0.20]{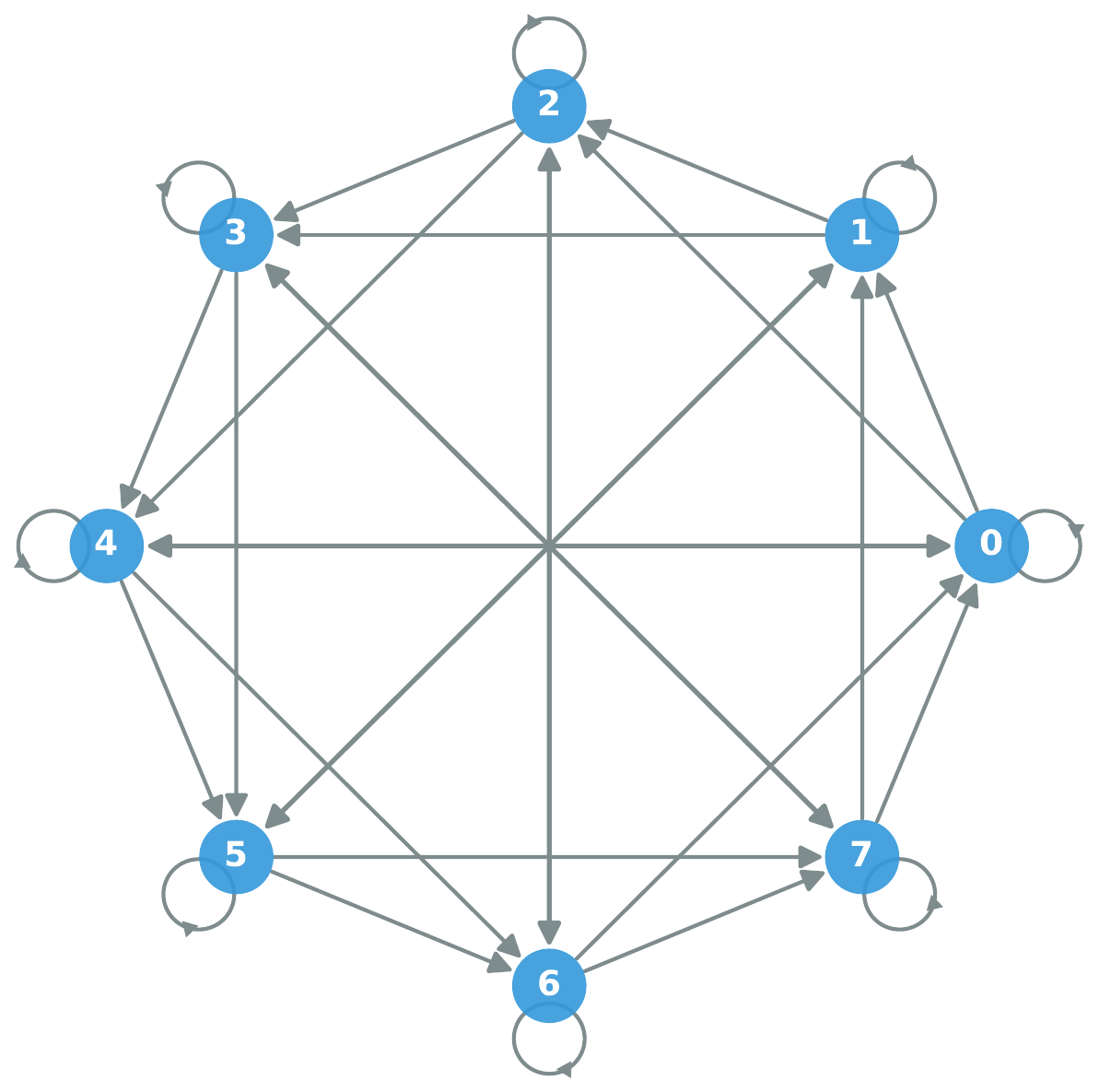} \\   
       \small (a) directed ring graph & 
       \small (b) directed exponential graph
  \end{tabular}
  \caption{We present the topologies of the directed ring graph with $n=16$ and the directed exponential graph with $n=8$.}
  \label{fig:topology}
\end{figure}

\section{Experiments}\label{sec:experiments}
In this section, we conduct the numerical experiments to validate the effectiveness of the proposed methods over directed networks.
The results on undirected networks are deferred to Appendix~\ref{appendix:exp-undirected}.

We conduct our experiments on the task of training a 6-layer Transformer-XL network
on the Penn Treebank (PTB) dataset that contains roughly one million tokens with a 10K vocabulary \citep{mikolov2012context}.
The architecture of the network in our experiments follows the setting of \citet{dai2019transformer}.

Specifically, each layer consists of 512-dimensional hidden states, 8 attention heads, and a 2048-dimensional feedforward network. 
We use the dropout with a rate of 0.2 on both attention and feedforward layers and  
set the segment length as 35 tokens.

We compare proposed decentralized normalized stochastic gradient descent with \pulldiag gradient tracking (NSGD-PD) and \pulldiag-GT with multiple gossips (MG-\pulldiag-GT) \cite{liang2025rowstochastic}.
We perform our experiments on the networks associated with the directed ring graph ($n=16$, $\beta = 0.952$, and $\kappa = 2$) and the directed exponential graph ($n=8$, $\beta=0.5$, and $\kappa = 1$) to evaluate the sample complexity and the communication complexity.
We present these topologies in Figure \ref{fig:topology}.
We set the number of consensus steps as $K \in \{3,5,8\}$ for the directed ring graph and $K \in \{1,3,5\}$ for the directed exponential graph, since the connectivity of the directed ring graph is weaker.
For all algorithms, we tune the step size from $\{0.005,0.01,0.02,0.05,0.1,0.2,0.3,0.4,0.5\}$ and fix the minibatch size as 20. 
We additionally run our NSGD-PD with $K=5$ and $n\in\{1,2,4,8\}$ to evaluate its speedup with respect to the number of agents.

We present empirical results in Figures~\ref{fig:directed-ring} and~\ref{fig:directed-ex}. 
Specifically, Figures~\ref{fig:directed-ring}(a),  \ref{fig:directed-ring}(b), \ref{fig:directed-ex}(a), and \ref{fig:directed-ex}(b) demonstrate that the proposed DNSGD-PD significantly outperforms the baseline method. 
These results validate our theoretical analysis, showing that the normalization step effectively mitigates the impact of heavy-tailed gradient noise.
Moreover, as shown in Figures~\ref{fig:directed-ex}(a) and \ref{fig:directed-ex}(b), increasing the parameter $K$ for our DNSGD-PD does not lead to improvement for the directed exponential graph, which is expected given the strong connectivity of this topology.
Additionally, Figures~\ref{fig:directed-ring}(c) and \ref{fig:directed-ex}(c) demonstrate our DNSGD-PD enjoys the linear speedup with respect to the number of agents $n$, matching our theoretical analysis in Section \ref{sec:main}.

\begin{figure}[t]
  \centering
  \begin{tabular}{c@{\hspace{2mm}}c@{\hspace{2mm}}c}
    \includegraphics[width=0.3\columnwidth]{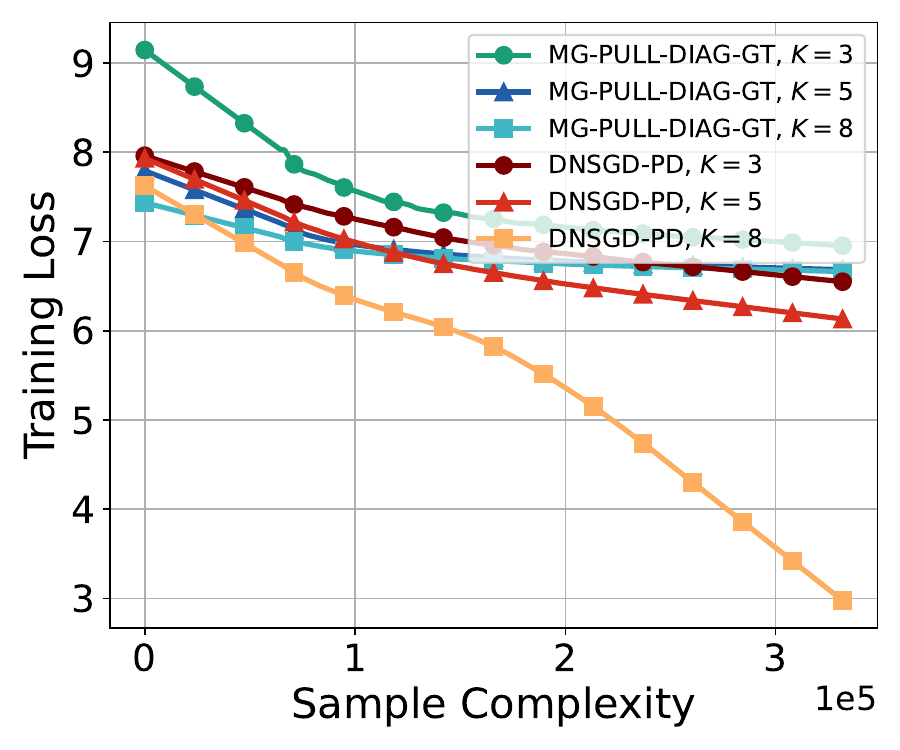} &
    \includegraphics[width=0.3\columnwidth]{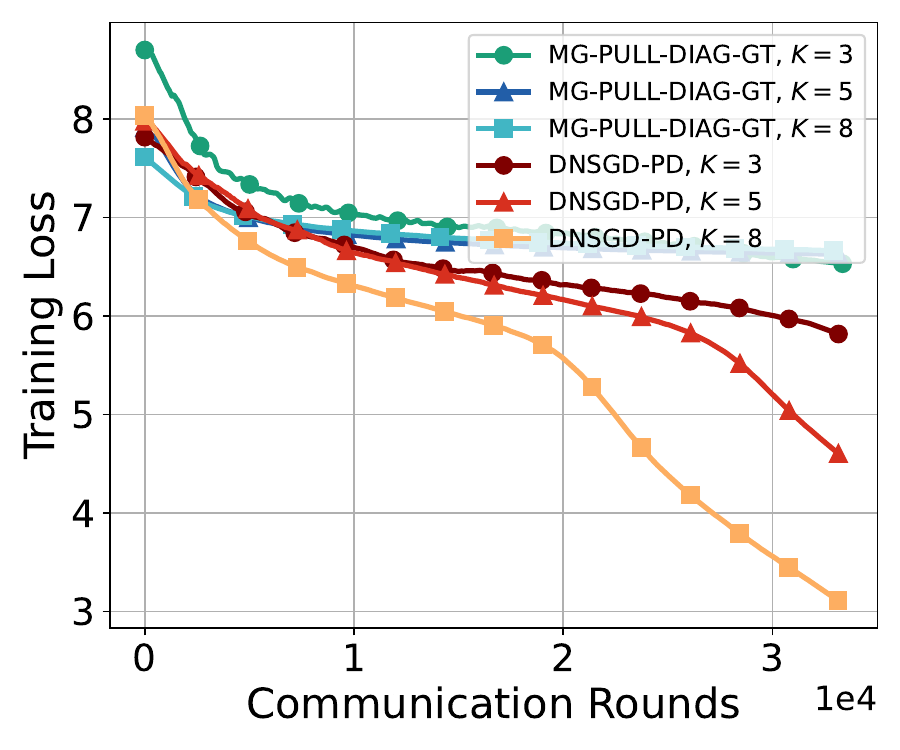} &
    \includegraphics[width=0.3\columnwidth]{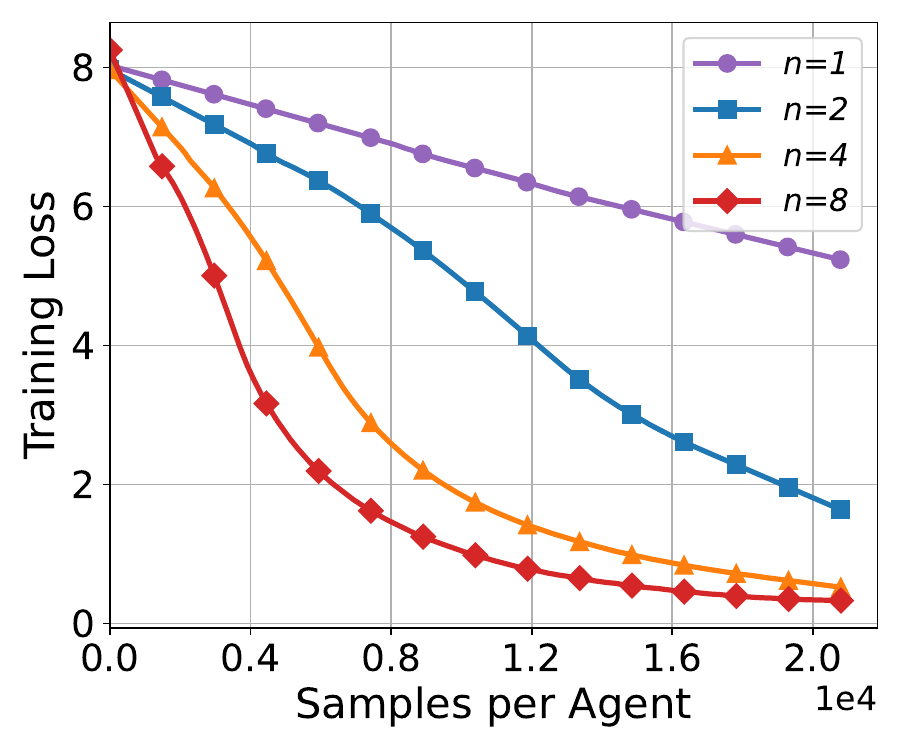} \\
    \small (a) \#sample vs. training loss &
    \small (b) \#communication vs. training loss &
    \small (c) speedup with respect to $n$
  \end{tabular} \vskip-0.18cm
  \caption{The comparison between proposed DNSGD-PD and baseline method MG-\pulldiag-GT on the directed ring network, 
  where subfigures (a) and (b) set $n=16$ and $K\in\{3,5,8\}$, 
  and subfigure (c) sets $n\in\{1,2,4,8\}$ and $K=5$.}
  \label{fig:directed-ring}
\end{figure}

\begin{figure}[t]
  \centering
  \begin{tabular}{c@{\hspace{2mm}}c@{\hspace{2mm}}c}
    \includegraphics[width=0.3\columnwidth]{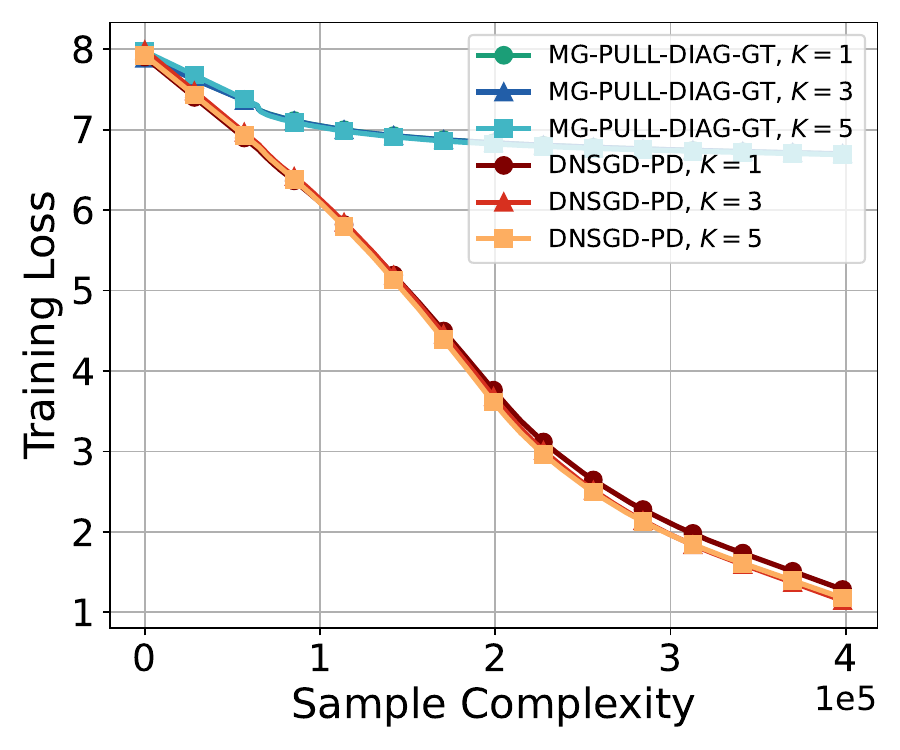} &
    \includegraphics[width=0.3\columnwidth]{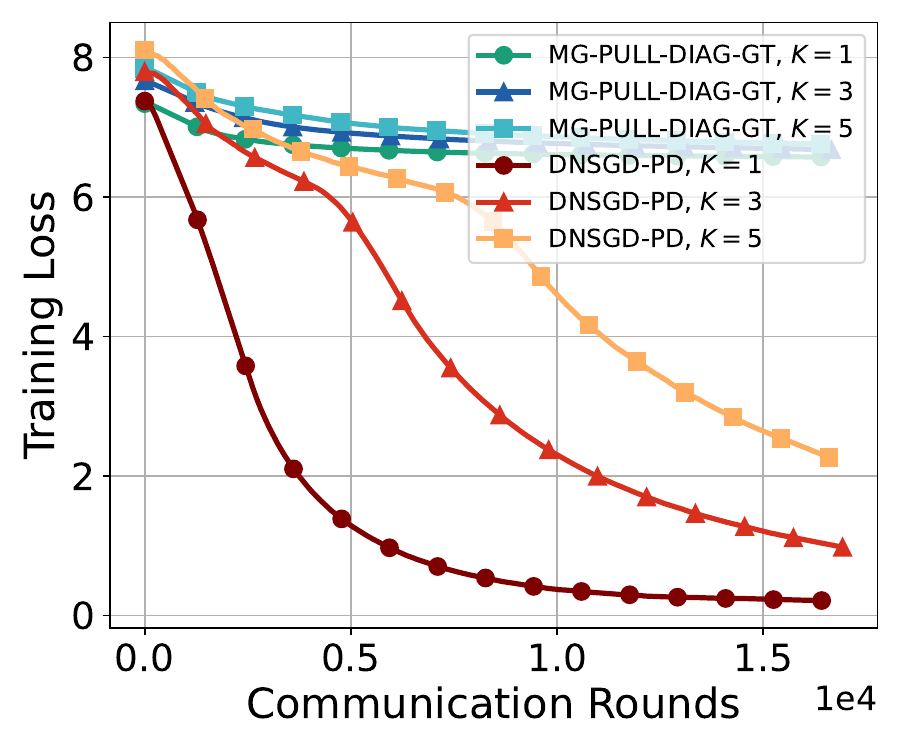} &
    \includegraphics[width=0.3\columnwidth]{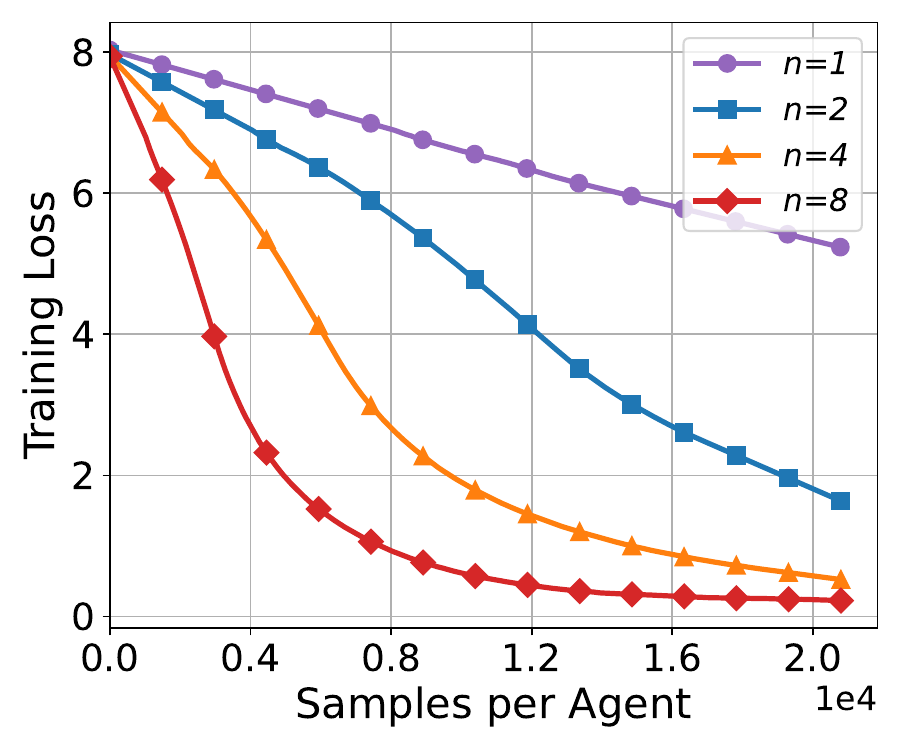} \\
    \small (a) \#sample vs.\ training loss &
    \small (b) \#communication vs.\ training loss &
    \small (c) speedup with respect to $n$
  \end{tabular}\vskip-0.18cm
  \caption{The comparison between the proposed DNSGD-PD and the baseline method MG-\pulldiag-GT on the directed exponential network, 
  where subfigures (a) and (b) set $n=8$ and $K\in\{1,3,5\}$, 
  and subfigure~(c) sets $n\in\{1,2,4,8\}$ and $K=5$.}
  \label{fig:directed-ex}
  \vskip -0.3cm
\end{figure}

\section{Conclusion}

In this paper, we study decentralized stochastic nonconvex optimization with heavy-tailed noise.
For the row-stochastic network, we propose decentralized normalized stochastic gradient descent with \pulldiag gradient tracking, which enjoys the linear speedup with respect to the number of agents.
Our theoretical analysis shows that the proposed method achieves both the optimal overall sample complexity and the near-optimal communication complexity.
We further follow our framework to study decentralized normalized stochastic gradient descent for nonconvex optimization over undirected networks, also achieving the optimal overall sample complexity and the near-optimal communication complexity under the heavy-tailed noise.
The experimental results on training Transformer-XL network show that the proposed method outperforms the baseline, which supports our theoretical results.

In future work, we would like to consider the conditions of the high-order smoothness and the averaged smoothness \citep{sun2024gradient,he2025complexity} to establish the sharper complexity bounds. 
It is also possible to extend our ideas to design decentralized algorithms for convex problem \citep{sun2024distributed} and column-stochastic networks \citep{liang2025understanding,assran2019stochastic,kungurtsev2023decentralized}.

\bibliographystyle{plainnat}
\bibliography{reference}

\appendix

\newpage
\appendix
\onecolumn

In appendices, we always follow notations in Section \ref{sec:pre}, i.e.,
\begin{align*}
{\bf\Pi}:= \operatorname{diag}(\vpi)\in\BR^{n\times n}
\qquad\text{and} \qquad
\Norm{\mA}_{\vpi} := \|{\bf\Pi}^{1/2}\mA{\bf\Pi}^{-1/2}\|_2.
\end{align*}
where
\begin{align*}
\vpi=[\pi_1,\dots,\pi_n]^\top \in \mathbb{R}^n    
\end{align*}
is the equilibrium vector defined in Proposition \ref{prop:Perron-Frobenius}.
We also characterize the mixing matrix $\mA$ by the parameters
\begin{align*}
    \beta:=\|\mA-\mathbf{1 } \vpi^{\top}\|_{\vpi}, \qquad
    \kappa:=\frac{\max_{i\in[n]} \pi_i}{\min_{i\in[n]} \pi_i} \in[1,+\infty), \qquad\text{and}\qquad \theta:=2n\kappa.
\end{align*}

\section{Supporting Lemmas}\label{sec:supporting-lemmas}

We first introduce some basic lemmas for addressing the stochastic gradient under the $p$-BCM assumption.

\begin{lem}[{\citet[Lemma 7]{hubler2024gradient}}]\label{lem: baseineq}
    For all $\va, \vb \in \mathbb{R}^n$ with $\vb \neq 0$, we have 
    \begin{align*}
        \frac{\va^{\top} \vb}{\Norm{\vb}} \geq\Norm{\va}-2\Norm{\va-\vb}.
    \end{align*}
\end{lem}

\begin{lem}[{\citet[Lemma 4.3]{liu2024nonconvex}}]\label{lem:heavy-tailed_cite}  Given a sequence of integrable random vectors $\vv_k\in \mathbb{R}^n,$ for all $k \in [b]$ such that $\mathbb{E}[\vv_k|\mathcal{F}_{k-1}]=0$ where $\mathcal{F}_k \triangleq \sigma\left(\vv_{1}, \cdots, \vv_{k}\right)$ is the natural filtration, then for all $p \in [1,2]$, it holds
\begin{align*}
    \mathbb{E}\left[\Norm{\sum_{k=1}^{b} \vv_{k}}\right] \leq 2 \sqrt{2} \mathbb{E}\left[\left(\sum_{k=1}^{b}\Norm{\vv_{k}}^p\right)^{\frac{1}{p}}\right].
\end{align*}
\end{lem}

\begin{lem}\label{lem:heavy-tailed} 

Under Assumption \ref{asm:pBCM}, let $\vg_{t,i}=\frac{1}{b} \sum_{k=1}^{b} \nabla F_{i}(\vx_{t,i} ; \xi_{t,i, k})$, then it holds that
\begin{align*}
    \BE[\Norm{\vg_{t,i}-\nabla f_{i}(\vx_{t,i})}] \leq   \frac{2 \sqrt{2}\sigma}{b^{1-1 / p}}
\end{align*}
for all $i\in[n]$.
\end{lem}
\begin{proof}
We let $\vdelta_{k}:=\frac{1}{b}\left(\nabla F_{i}(\vx_{t,i} ; \xi_{t,i,k})-\nabla f_{i}(\vx_{t,i})\right)$.
According to Lemma \ref{lem:heavy-tailed_cite} with $\vv_k=\vdelta_k$, we have
\begin{align}\label{eq:exp-dk}
    &\mathbb{E}\left[\Norm{\sum_{k=1}^{b} \vdelta_{k}}\right] 
    \leq 2 \sqrt{2} \mathbb{E}\left[\left(\sum_{k=1}^{b}\Norm{\vdelta_{k}}^{p}\right)^{1 / p}\right]= 2 \sqrt{2} \mathbb{E}\left [\left(\frac{1}{b^{p}} \sum_{k=1}^{b}\Norm{\nabla F_{i}(\vx_i ; \xi_{t,i,k})-\nabla f_{i}(\vx_i)}^{p}\right)^{1 / p}\right]. 
\end{align}
Then it holds
\begin{align*}    
\begin{split}    
& \BE\left[\Norm{\vg_{t,i}-\nabla f_{i}(\vx_{t,i})}\right]  
=\mathbb{E}\left[\Norm{\sum_{k=1}^{b} \vdelta_k}\right] \\
 \leq & 2 \sqrt{2} \mathbb{E}\left[\frac{1}{b}\left(\sum_{k=1}^{b}\Norm{\nabla F_{i}(\vx_{t,i} ; \xi_{t,i, k})-\nabla f_{i}(\vx_{t,i})}^{p}\right)^{1 / p}\right]\\
= &\frac{2 \sqrt{2}}{b} \mathbb{E}\left[\left(\sum_{k=1}^{b}\Norm{\nabla F_{i}(\vx_{t,i} ; \xi_{t,i, k})-\nabla f_{i}(\vx_{t,i})}^{p}\right)^{1 / p}\right]\\
\leq & \frac{2 \sqrt{2}}{b} \left(\mathbb{E}\left[\sum_{k=1}^{b}\Norm{\nabla F_{i}(\vx_{t,i} ; \xi_{t,i, k})-\nabla f_{i}(\vx_{t,i})}^{p}\right]\right)^{1 / p}\\
   \leq  & \frac{2 \sqrt{2}}{b} \cdot b^{1 / p} \sigma= \frac{2 \sqrt{2}\sigma}{b^{1-1 / p}},
\end{split}
\end{align*}
where the first inequality is based on equation (\ref{eq:exp-dk}), the second inequality is based on Jensen's inequality and the last inequality is based on Assumption \ref{asm:pBCM}.
\end{proof}

We then provide some results for communication protocol associated with the row-stochastic matrix $\mA\in\BR^{n\times n}$.

\begin{lem}[{\citet[Lemma 12]{liang2025rowstochastic}}]\label{lem:DiagnoseError} 
Suppose the matrix $\mA\in\BR^{n\times n}$ satisfies Assumption \ref{asm:primitive}, then the diagonal matrix $\mD_{t}=\operatorname{Diag}(\mA^{(t+1)K})$ holds
\begin{align*}
    \Norm{\vpi^{\top} \mD_{t}^{-1}-\mathbf{1}_{n}^{\top}} \leq \theta\sqrt{n \kappa} \beta^{(t+1)K}
\end{align*}
for all $K\geq 1$, where $\theta=2n\kappa$.
\end{lem}

\begin{lem}\label{lem:A_iterator}
    Following notations and assumptions of Proposition \ref{prop:Perron-Frobenius}, we have 
    \begin{align}\label{recursion:AKpi}
    \mA^K-\boldsymbol{1}\vpi^\top = (\mA-\boldsymbol{1}\vpi^\top)^K.
    \end{align}
\begin{proof}
We prove this result by induction. 
\begin{enumerate}
    \item For the induction base $K=1$, equation (\ref{recursion:AKpi})  clearly holds.
    \item For the induction step, we suppose equation (\ref{recursion:AKpi}) holds for all $K=m\geq 1$. Then for the case of $K=m+1$, we have 
    \begin{align*}
          & \mA^{m+1}-\boldsymbol{1}\vpi^\top \\
        =& \mA^{m+1}-\boldsymbol{1}\vpi^\top \mA=(\mA^m-\boldsymbol{1}\vpi^\top)\mA \\
        =&(\mA-\boldsymbol{1}\vpi^\top)^m\mA=(\mA-\boldsymbol{1}\vpi^\top)^m(\mA-\boldsymbol{1}\vpi^\top+\boldsymbol{1}\vpi^\top)\\
        =&(\mA-\boldsymbol{1}\vpi^\top)^{m+1}+(\mA-\boldsymbol{1}\vpi^\top)^m\boldsymbol{1}\vpi^\top, 
    \end{align*}
    where the second line is based on the fact $\vpi^{\top} \mA=\vpi^{\top}$ from Proposition \ref{prop:Perron-Frobenius};
    the third line is based on the induction base.
    For the last term in above inequality, we have
    \begin{align*}
         & (\mA-\boldsymbol{1}\vpi^\top)^m\boldsymbol{1}\vpi^\top\\
        =&(\mA-\boldsymbol{1}\vpi^\top)^{m-1}(\mA-\boldsymbol{1}\vpi^\top)\boldsymbol{1}\vpi^\top\\
        =&(\mA-\boldsymbol{1}\vpi^\top)^{m-1}(\mA \boldsymbol{1}\vpi^\top-\boldsymbol{1}\vpi^\top\boldsymbol{1}\vpi^\top)\\
        =&(\mA-\boldsymbol{1}\vpi^\top)^{m-1}(\boldsymbol{1}\vpi^\top-\boldsymbol{1}\vpi^\top)=\vzero,
    \end{align*}
    where the last line is based on the relations $\mA\vone=\vone$ and $\vpi^\top\vone=1$ from Assumption \ref{asm:primitive} and Proposition \ref{prop:Perron-Frobenius}.
    Combing above results leads to 
    \begin{align*}
        \mA^{m+1}-\boldsymbol{1}\vpi^\top=(\mA-\boldsymbol{1}\vpi^\top)^{m+1},
    \end{align*}
    which completes the proof.
\end{enumerate}
\end{proof}    
\end{lem}

\begin{lem}\label{lem: matrixerror} Suppose the matrix $\mA\in\BR^{n\times n}$ satisfies Assumption \ref{asm:primitive}, for all $K \geq 0$, we have
\begin{align*}
    \Norm{\mA^{K}-\mathbf{1 } \vpi^{\top}}\leq \sqrt{n\kappa} \beta^{K}.
\end{align*}
\end{lem}
\begin{proof}
Recall that ${\bf\Pi}:= \operatorname{diag}(\vpi)\in\BR^{n\times n}$, then we have
\begin{align*}
\begin{split}    
    \Norm{\mA^{K}-\mathbf{1 } \vpi^{\top}}
=& \Norm{\mPi^{-1 / 2}\left(\mPi^{1 / 2} (\mA^{K}-\mathbf{1 } \vpi^{\top}) \mPi^{-1 / 2}\right) \mPi^{1 / 2}} \\
\leq &  \big\|\mPi^{-1 / 2}\big\|_{2}\big\|\mPi^{1 / 2}\big\|_{2}\big\|\mPi^{1 / 2} (\mA^{K}-\mathbf{1 } \vpi^{\top}) \mPi^{-1 / 2}\big\|\\
\leq &\frac{\sqrt{\operatorname{max}_{i\in [n]}\pi_i}}{\sqrt{\operatorname{min}_{i\in [n]}\pi_i}} \cdot \sqrt{n}\Norm{\mPi^{1 / 2}(\mA^{K}-\mathbf{1 } \vpi^{\top}) \mPi^{-1 / 2}}_{2}\\
=& \sqrt{n\kappa}\Norm{\mA^{K}-\mathbf{1 } \vpi^{\top}}_{\vpi} 
= \sqrt{n\kappa}\Norm{(\mA-\mathbf{1 } \vpi^{\top})^K}_{\vpi} \\
\leq & \sqrt{n\kappa}\Norm{\mA-\mathbf{1 } \vpi^{\top}}^{K}_{\vpi} 
= \sqrt{n\kappa} \beta^{K},
\end{split}
\end{align*}
where the first three lines are based on the property of spectral norm and the definitions of $\mPi$ and $\kappa$;
the second-to-last line is base on Lemma \ref{lem:A_iterator};
the last line is based on the fact
\begin{align*}
   & \Norm{(\mA-\mathbf{1 } \vpi^{\top})^K}_{\vpi}
= \Norm{\mPi^{1/2}(\mA-\mathbf{1 } \vpi^{\top})^K\mPi^{-1/2}}_2 \\
= & \Norm{(\mPi^{1/2}(\mA-\mathbf{1 } \vpi^{\top})\mPi^{-1/2})^K}_2 \\
\leq & \Norm{\mPi^{1/2}(\mA-\mathbf{1 } \vpi^{\top})\mPi^{-1/2}}_2^K
= \Norm{\mA-\mathbf{1 } \vpi^{\top}}_{\vpi}^{K},
\end{align*}
where the last line is based on the triangle inequality.
\end{proof}

\begin{lem}\label{lem: iteration error for D} 
Suppose the matrix $\mA\in\BR^{n\times n}$ satisfies Assumption \ref{asm:primitive} and let $\mD_{t}=\operatorname{Diag}(\mA^{(t+1)K})$, then we have 
\begin{align*}
    \Norm{\mD_{t}^{-1}-\mD_{t+1}^{-1}} \leq 2 \theta n^2\kappa^{3/2} \beta^{(t+1)K}.
\end{align*}
\end{lem}

\begin{proof}
According the property of spectral norm, we have
\begin{align*}
    \Norm{\mD_{t}^{-1}-\mD_{t+1}^{-1}} \leq \sqrt{n}\Norm{\mD_{t}^{-1}-\mD_{t+1}^{-1}}_{2}\leq \sqrt{n}2 \theta \sqrt{\kappa^{3} n^{3}} \beta^{(t+1)K}=2 \theta n^2\kappa^{3/2} \beta^{(t+1)K},
\end{align*}
where the second inequality is based on Lemma 12 of \citet{liang2025rowstochastic}.
\end{proof}

\begin{lem}\label{lem:diagestimate} 
Following the settings of Lemma \ref{lem: iteration error for D} with 
$K \geq 3(1+\ln (\kappa n))/(1-\beta)$, 
we have
\begin{align*}
\Norm{\vpi^\top\mD_t^{-1}} \leq \sqrt{2n}.    
\end{align*}
\end{lem}
\begin{proof}
Based on the triangle inequality, we have
\begin{align*}
    \Norm{\vpi^{\top} \mD_{t}^{-1}} &\leq \Norm{\mathbf{1^\top}}+\Norm{\vpi^{\top} \mD_{t}^{-1}-\mathbf{1}^{\top}} \\
    & \leq \sqrt{n}+\theta \sqrt{n \kappa} \beta^{(t+1)K} \\
&  \leq\sqrt{n}+\frac{2}{{\rm e}^3 n \kappa}\leq\sqrt{2n},
\end{align*}
where the second inequality is based on Lemma \ref{lem:DiagnoseError} and the third inequality is based on the setting of $K$.
\end{proof}

\begin{lem}[{\citet{ye2023multi}},Lemma 12]\label{lem:undirected_norm_ineq}
    For given matrix $\mG=[\vg_1;\cdots;\vg_m]\in \BR^{n\times d}$, we have $\Norm{\mG-\boldsymbol{1}\bar{\vg}}^2\leq \Norm{\mG}^2$, where $\bar{\vg}=\frac{1}{m}\sum_{i=1}^m\vg_i$.
\end{lem}

\section{The Proofs of Results in Section \ref{sec:main}}

This section provides the detailed proofs of theoretical results for proposed DNSGD-PD in Section \ref{sec:main}.

\subsection{The Proof of Proposition \ref{lem:Bounded Diagonals}}
\begin{proof}
We denote $K=r/(1-\beta)$ with $r>3(1+\ln (\kappa n))$, then it holds
\begin{align}\label{ine:bound of beta_K}
\beta^{K}=\exp \left(\frac{r\ln \beta}{1-\beta}\right) \leq \exp (-r)  \leq \frac{1}{{\rm e}^{3} n^{3} \kappa^{3}}.
\end{align}

Here, the first inequality is based on the fact
$\ln x \le x-1$ for all $x\in(0,1)$ which implies
$\ln\beta \le -(1-\beta)$, then it holds
\begin{align*}
    \frac{r\ln \beta}{1-\beta} \le -r,
\end{align*}
which implies
\begin{align*}
    \exp\left(\frac{r\ln \beta}{1-\beta}\right)
\le \exp(-r).
\end{align*}
Furthermore, diagonal entries of $\mA^K$ satisfies
\begin{align*}
\begin{split}
    [\mA^{K}]_{i i}&=\pi_i+[\mA^{K}-\mathbf{1 } \vpi^{\top}]_{i i} \\
    & \geq \operatorname{min}_{j\in[n]}\pi_j-\Norm{\mA^{K}-\mathbf{1 } \vpi^{\top}}_{F} \\
    &\geq \operatorname{min}_{j\in[n]}\pi_j-\sqrt{n \kappa} \beta^{K}\\
    &\geq \frac{1}{n \kappa}-\sqrt{n \kappa} \exp (-r) \\
    &\geq \frac{1}{n \kappa}-\frac{1}{2 n \kappa}=\frac{1}{2 n \kappa}, 
\end{split}
\end{align*}
for all $i\in[n]$, 
where the first inequity is based on the definition of Frobenius norm;
the second inequality is based on Lemma \ref{lem: matrixerror};
the third inequality is based on equation (\ref{ine:bound of beta_K}) and the fact $\operatorname{max}_{j\in[n]}\pi_j\geq 1/n$ that leads to
\begin{align*}
    \operatorname{min}_{j\in[n]}\pi_j
\geq \frac{\operatorname{min}_{j\in[n]}\pi_j}{n\operatorname{max}_{j\in[n]}\pi_j} = \frac{1}{n\kappa};
\end{align*}
and the last line is based on the definition of $r$.
\end{proof}

\subsection{The Proof of Lemma \ref{lem:descent-mean}}
\begin{proof}
Recall that $\vw_{t} = \vpi^{\top} \mX_{t}$, then left-multiply $\vpi^\top$ on both sides of line \ref{line:update-X} in Algorithm \ref{alg:PDNSGD} leads to
\begin{align*}
\vw_{t+1}=\vw_{t}-\eta \vpi^{\top}  \mU_{t}.     
\end{align*}
According to above equation and Assumption \ref{asm:smoothness}, we have
\begin{align}\label{neq:smoothness}  
    f(\vw_{t+1}) \leq f(\vw_t)- \eta\left\langle \nabla f(\vw_{t}), \vpi^{\top} \mU_{t}\right\rangle+\frac{\eta^{2} L}{2}\Norm{\vpi^{\top} \mU_{t}}^{2}.
\end{align}
We then focus on the inner product term on the right-hand side of (\ref{neq:smoothness}) as follows
\begin{align}\label{neq:inner} 
\begin{split}    
&- \eta\left\langle \nabla f(\vw_{t}), \vpi^{\top} \mU_{t}\right\rangle \\
&=-\frac{\eta}{n}\left\langle n\nabla f\left(\vw_{t}\right), \sum_{i=1}^n \frac{\pi_i\vv_{t,i}}{\Norm{\vv_{t,i}}}\right\rangle \\
&=-\frac{\eta}{n} \left(\sum_{i=1}^n \frac{\pi_{i} n\nabla f(\vw_{t})^\top\vv_{t,i}}{\Norm{\vv_{t,i}}}\right)\\
&\leq  -\frac{\eta}{n}\left(\sum_{i=1}^n \pi_{i}  (\Norm{n\nabla f(\vw_{t})}-2\Norm{n\nabla f(\vw_{t})-\vv_{t,i}})\right)\\
&=-\frac{\eta}{n} \left( \Norm{n\nabla f(\vw_{t})}-2 \sum_{i=1} ^n\pi_{i}\Norm{n\nabla f(\vw_{t})-\vv_{t,i}} \right)\\
&\leq -\eta \Norm{\nabla f(\vw_{t})}  + \frac{2\eta}{n} \left( \sum_{i=1}^n \pi_{i}\Norm{n\nabla f(\vw_{t})-\vpi^{\top} \mV_{t}}+\sum_{i=1}^n \pi_{i}\Norm{\vpi^{\top} \mV_{t}-\vv_{t,i}} \right)\\
&=-\eta \Norm{\nabla f(\vw_{t})}  + \frac{2\eta}{n} \left(\Norm{n\nabla f(\vw_{t})-\vpi^{\top} \mV_{t}}+\sum_{i=1}^n \pi_{i}\Norm{\vpi^{\top} \mV_{t}-\vv_{t,i}}\right)\\
& \leq -\eta \Norm{\nabla f(\vw_{t})}  + \frac{2\eta}{n} \left( \Norm{n\nabla f(\vw_{t})-\vpi^{\top} \mV_{t}}+ \Norm{\mV_{t}-\mathbf{1} \vpi^{\top} \mV_{t}} \right),
\end{split}
\end{align}
where the equalities uses the fact $\sum_{i=1}^n\pi_i=1$; the first inequality is based on Lemma \ref{lem: baseineq}; 
the second inequality is based on triangle inequality; 
the last step uses Cauchy--Schwarz inequality and the fact $\sum_{i=1}^n\pi_i^2\leq 1$ which leads to 
\begin{align*}
    \sum_{i=1}^n \pi_{i}\Norm{\vpi^{\top} \mV_{t}-\vv_{t,i}}\leq \sqrt{\sum_{i=1}^n\pi_i^2}\sqrt{\sum_{i=1}^n \Norm{\vpi^\top\mV_t-\vv_{t,i}}^2} 
    \leq \Norm{\mathbf{1} \vpi^{\top} \mV_{t}-\mV_{t}}.
\end{align*}
Combing equations (\ref{neq:smoothness}) and (\ref{neq:inner}), we obtain 
\begin{align*}
     f(\vw_{t+1})
    & \leq f(\vw_{t})-\eta \Norm{\nabla f(\vw_{t})}  + \frac{2\eta}{n}  \left(\Norm{n\nabla f(\vw_{t})-\vpi^{\top} \mV_{t}}+ \Norm{\mV_{t}-\mathbf{1} \vpi^{\top} \mV_{t}}  \right) +\frac{\eta^{2} L}{2}\Norm{\vpi^{\top} \mU_{t}}^{2}\\
    &\leq f(\vw_{t})-\eta \Norm{\nabla f(\vw_{t})} + \frac{2\eta}{n} \Norm{n\nabla f(\vw_{t}) -\vpi^{\top} \mV_{t}} + \frac{2\eta}{n}\Norm{\mV_{t}-\mathbf{1} \vpi^{\top} \mV_{t}}+\frac{\eta^{2} L}{2},
\end{align*}
where the last step is based on facts that the norm of each row of $\mU_t$ is one and $\sum_{i=1}^n\pi_i^2\leq 1$.
Taking expectations on both sides of the above inequality, then we finish the proof.
\end{proof}

\subsection{The Proof of Lemma \ref{lem:Estimate descent deviation}}
\begin{proof}
For the update of $\mV_{t+1}$ in Algorithm \ref{alg:PDNSGD}, we have
\begin{align*}
    \mV_{t+1} = \mA^{K}(\mV_{t} + \mD_{t+1}^{-1} \mG_{t+1}  - \mD_{t}^{-1} \mG_{t}).
\end{align*}
Left-multiplying $\vpi^\top$ on both sides of above equation yields
\begin{align*}
&\vpi^\top\mV_{t+1}- \vpi^\top\mD_{t+1}^{-1} \mG_{t+1} \\ 
= & \vpi^\top\left(\mA^{K}(\mV_{t} + \mD_{t+1}^{-1} \mG_{t+1}  - \mD_{t}^{-1} \mG_{t}) - \mD_{t+1}^{-1} \mG_{t+1}\right) \\
= & \vpi^\top\mV_{t} + \vpi^\top\mD_{t+1}^{-1} \mG_{t+1}  - \vpi^\top\mD_{t}^{-1} \mG_{t} - \vpi^\top\mD_{t+1}^{-1} \mG_{t+1} \\
= & \vpi^\top\mV_{t} - \vpi^\top\mD_{t}^{-1} \mG_{t}\\
= & \cdots \\
= & \vpi^\top\mV_0- \vpi^\top\mD_0^{-1} \mG_0 \\
= & \vpi^\top \mA^{\hat{K}} \mD_0^{-1}\mG_0-\vpi^\top\mD_0^{-1} \mG_0 \\
= & \vzero,
\end{align*}
where the second equality is based on the fact $\vpi^{\top} \mA=\vpi^{\top}$ from Proposition \ref{prop:Perron-Frobenius};
the second-to-last line is based on the setting of~$\mV_0=\mA^{\hat{K}} \mD_0^{-1}\mG_0$;
and the last step again uses the fact $\vpi^{\top} \mA=\vpi^{\top}$ from Proposition \ref{prop:Perron-Frobenius}.
Therefore, for all $t\geq 0$, it holds 
\begin{align}\label{eq:piVDG}
 \vpi^{\top}\mV_t=\vpi^{\top} \mD_{t}^{-1} \mG_{t}.
\end{align}

We now consider the statement of this lemma. According equation (\ref{eq:piVDG}) and the triangle inequality, we split the upper bound of $\Norm{n\nabla f(\vw_{t})-\vpi^{\top} \mV_t}$ as follows
\begin{align}\label{eq:diff-grad-V}
\begin{split}
& \Norm{n\nabla f(\vw_{t})-\vpi^{\top} \mV_t} = \Norm{n\nabla f(\vw_{t})-\vpi^{\top} \mD_{t}^{-1} \mG_{t}}    \\
\leq & \underbrace{\Norm{n\nabla f(\vw_{t})-\mathbf{1}^{\top} \nabla \mF(\mX_{t})}}_{A_{1}} 
+\underbrace{\Norm{\mathbf{1}^\top\nabla \mF(\mX_{t})-\vpi^{\top}\mD_{t}^{-1} \nabla \mF(\mX_{t})}}_{A_{2}}
+\underbrace{\Norm{\vpi^{\top} \mD_{t}^{-1}(\nabla \mF(\mX_{t})-\mG_{t})}}_{A_3}. 
\end{split}
\end{align}
For the term $A_1$, it holds
\begin{align}\label{eq:A_1}
\begin{split}
    & A_{1}=\Norm{n\nabla f(\vw_{t})- \sum_{i=1}^{n} \nabla f_{i}(\vx_{t,i})}= \Norm{\sum_{i=1}^{n}(\nabla f_{i}(\vw_{t})-\nabla f_{i}(\vx_{t,i}))}\\
    &\leq L\sum_{i=1}^{n}\Norm{\vx_{t,i}-\vw_{t}}
     \leq L\sqrt{n}\sqrt{\sum_{i=1}^{n}\Norm{\vx_{t,i}-\vw_{t}}^{2}} \\
    & =L\sqrt{n}\Norm{\mX_{t}-\mathbf{1} \vw_{t}} = L\sqrt{n}\Norm{\mX_{t}-\mathbf{1}\vpi^\top \mX_{t}},
\end{split}
\end{align}
where the fist inequality is based on the Assumption \ref{asm:smoothness} and the last inequality is based on the Cauchy–Schwarz inequality.

For the term term $A_2$, it holds
\begin{align}\label{eq:A_2}
\begin{split}
    A_2&=\Norm{(\mathbf{1}^\top-\vpi^{\top}\mD_{t}^{-1}) \nabla \mF(\mX_{t})}\\
&\leq \Norm{\mathbf{1}^\top-\vpi^{\top}\mD_{t}^{-1}}\Norm{\nabla \mF(\mX_{t})}\\
& \leq \theta\sqrt{n \kappa}\beta^{(t+1)K}\Norm{\nabla \mF(\mX_{t})}\\
&\leq \theta\sqrt{n \kappa} \beta^{(t+1)K} \left( \Norm{\nabla \mF(\mX_{t})-\nabla \mF(\mathbf{1}\vw_{t})}+\Norm{\nabla \mF(\mathbf{1}\vw_{t})} \right)\\
&\leq \theta\sqrt{n \kappa} \beta^{(t+1)K} \left(L\Norm{\mX_{t}-\mathbf{1} \vpi^{\top} \mX_{t}}+ \Norm{\nabla \mF(\mathbf{1}\vw_{t})}\right)\\
&\leq \theta\sqrt{n \kappa} \beta^{(t+1)K} \left(L\Norm{\mX_{t}-\mathbf{1}\vpi^{\top} \mX_{t}}+\sqrt{n} \Norm{\nabla f(\vw_{t})}\right),
\end{split}
\end{align}
where the second inequality is based on Lemma \ref{lem:DiagnoseError}; the fourth inequality is base on the Assumption \ref{asm:smoothness};
the other steps are based on triangle inequalities and the definition of Frobenius norm.

For the term $A_3$, we split its upper bounds as follows
\begin{align}\label{eq:A3-1}
    A_3 =  \underbrace{\Norm{(\vpi^{\top} \mD_{t}^{-1}-\boldsymbol{1}^\top)(\nabla \mF(\mX_{t})-\mG_{t})}}_{B_1}  +\underbrace{\Norm{\boldsymbol{1}^\top(\nabla \mF(\mX_{t})-\mG_{t})}}_{B_2}.
\end{align}
We bound the term $B_1$ as 
\begin{align}\label{eq:B1}
\begin{split}
        \BE[B_1]
    &\leq \Norm{\vpi^{\top} \mD_{t}^{-1}-\boldsymbol{1}^\top}\BE\left[\Norm{\nabla \mF(\mX_{t})-\mG_{t}}\right] \\
    &\leq \theta\sqrt{n\kappa}\beta^{(t+1)K} \BE\left[\sqrt{n}\max_{i\in[n]}\Norm{\vg_{t,i}-\nabla f_{i}(\vx_{t,i})}\right]\\
    & \leq  \theta\sqrt{n\kappa}\beta^{(t+1)K} \cdot \frac{2 \sqrt{2n}\sigma}{b^{1-1 / p}},
\end{split}
\end{align}
where the second inequality is due to Lemma \ref{lem:DiagnoseError} and the last inequality follows from Lemma \ref{lem:heavy-tailed}.
We then bound the term~$B_2$ as
\begin{align}\label{eq:B2}
\begin{split}
    \BE[B_2] &= \BE\left[\Norm{\sum_{i=1}^n \left( \nabla f_i(\vx_{t,i}) - \vg_{t,i}\right)}\right]  \\
    &=\BE\left[\Norm{\sum_{i=1}^n \left(\nabla f_i(\vx_{t,i}) - \frac{1}{b}\sum_{k=1}^b\nabla F_i(\vx_{t,i};\xi_{t,i,k})\right)}\right]  \\
    &\leq 2\sqrt{2}\BE \left[\left(\sum_{i=1}^n\Norm{\nabla f_i(\vx_{t,i}) - \frac{1}{b}\sum_{k=1}^b\nabla F_i(\vx_{t,i};\xi_{t,i,k})}^p\right)^\frac{1}{p}\right]\\
   &\leq 2\sqrt{2} \left(\BE\left[\sum_{i=1}^n\Norm{\nabla f_i(\vx_{t,i}) - \frac{1}{b}\sum_{k=1}^b\nabla F_i(\vx_{t,i};\xi_{t,i,k})}^p\right]\right)^\frac{1}{p} \\
   &\leq 2\sqrt{2}\left(n \left(\frac{2\sqrt{2}\sigma}{b^{1-1/p}}\right)^p\right)^{\frac{1}{p}} = \frac{8n^{\frac{1}{p}}\sigma}{b^{1-1/p}},
\end{split}
\end{align}
where the first inequality is based on Lemma \ref{lem:heavy-tailed_cite};
the second inequality is based on the Jensen’s inequality;
and the last inequality is based on Lemma \ref{lem:heavy-tailed}.
Combining equations (\ref{eq:B1}) and (\ref{eq:B2}) with equation (\ref{eq:A3-1}), we obtain 
\begin{align}\label{eq:A_3}
    \BE[A_3] \leq  \frac{2\sqrt{2\kappa}n\theta\beta^{(t+1)K}\sigma}{b^{1-1 / p}}+\frac{8n^{\frac{1}{p}}\sigma}{b^{1-1/p}}.
\end{align}
Combining equations (\ref{eq:A_1}), (\ref{eq:A_2}), and (\ref{eq:A_3}) with equation (\ref{eq:diff-grad-V}), we  completes the proof.
\end{proof}

\subsection{The Proof of Lemma \ref{lem:consensus_x}}
\begin{proof} 

According to equations $\mA\vone=\vone$ and $\vpi^\top\vone=1$ from Assumption \ref{asm:primitive} and Proposition \ref{prop:Perron-Frobenius}, we have
\begin{align}\label{eq:eq1}
(\mA^{K}-\mathbf{1} \vpi^{\top})\boldsymbol{1}\vpi^\top
= (\mA^{K}\boldsymbol{1}\vpi^\top-\mathbf{1} (\vpi^{\top}\boldsymbol{1})\vpi^\top)
= \boldsymbol{1}\vpi^\top-\mathbf{1}\vpi^\top = \vzero.
\end{align}
Based on the update $\mX_{t+1}=\mA^{K}(\mX_{t}-\eta \mU_{t})$ in Algorithm \ref{alg:PDNSGD}, we have
\begin{align}\label{eq:X_t1_bound}
\begin{split}
& \Norm{\mX_{t+1}-\mathbf{1}\vpi^\top  \mX_{t+1}} \\
=& \Norm{(\mI-\mathbf{1} \vpi^{\top}) \mA^{K}(\mX_{t}-\eta \mU_{t})}\\
=&\Norm{(\mA^{K}-\mathbf{1} \vpi^{\top})(\mX_{t}-\eta \mU_{t})}\\
=&\Norm{(\mA^{K}-\mathbf{1} \vpi^{\top}) \mX_{t}-\eta(\mA^{K}-\mathbf{1}\vpi^{\top}) \mU_{t}}\\
=&\Norm{(\mA^{K}-\mathbf{1} \vpi^{\top}) (\mX_{t}-\mathbf{1} \vpi^{\top} \mX_{t})-\eta(\mA^{K}-\mathbf{1} \vpi^{\top}) (\mU_{t}-\mathbf{1} \vpi^{\top} \mU_{t})}\\
\leq  &\Norm{\mA^{K}-\mathbf{1} \vpi^{\top}}\Norm{\mX_{t}-\mathbf{1} \vpi^{\top} \mX_{t}-\eta(\mU_{t}-\mathbf{1} \vpi^{\top} \mU_{t})}\\
\leq & \rho\left(\Norm{\mX_{t}-\mathbf{1}\vpi^\top  \mX_{t}}+\eta\Norm{\mU_{t}-\mathbf{1}\vpi^\top  \mU_{t}}\right) \\
\leq & \rho\left(\Norm{\mX_{t}-\mathbf{1}\vpi^\top  \mX_{t}}+2\sqrt{n}\eta\right),
\end{split}
\end{align}
where the second equality is based on the fact $\vpi^\top\mA=\vpi^\top$ from Proposition \ref{prop:Perron-Frobenius};
the forth equality is based on equation (\ref{eq:eq1});
the first inequality is based on the triangle inequality;
the second inequality is based on the definition of $\rho$.
Additionally, the last inequality in above derivation is based on
\begin{align*}
    \Norm{\mU_{t}-\mathbf{1}\vpi^\top  \mU_{t}} \leq\Norm{\mU _t}+\Norm{\mathbf{1}\vpi^\top  \mU_{t}}
    =\sqrt{n}+\sqrt{ n}\Norm{\vpi^\top \mU_t}
    \leq\sqrt{n} +\sqrt{n}\leq 2\sqrt{n},
\end{align*}
where the first inequality is based on triangle inequality;
the second inequality is based on the facts the norm of each row of $\mU$ is one that leads to
\begin{align}\label{eq:piUt}
    \Norm{\vpi^\top \mU_t} 
= \Norm{\sum_{k=1}^n \pi_k\cdot\dfrac{ \vv_{t,1}}{\Norm{\vv_{t,1}}}} 
\leq \sum_{k=1}^n \pi_k \Norm{\dfrac{ \vv_{t,1}}{\Norm{\vv_{t,1}}}} 
= \sum_{k=1}^n \pi_k = 1.
\end{align}
where we the first step uses the definition of $\mU_t$ in Algorithm \ref{alg:PDNSGD}; the second step is based on triangle inequality; 
and the last step is based on Proposition \ref{prop:Perron-Frobenius}.
are based on the fact that . Hence, we finish the proof.
\end{proof}

\subsection{The Proof of Lemma \ref{lem:consensus_v}}
\begin{proof}
According to the update $\mV_{t+1}=\mA^{K}(\mV_{t}+\mD_{t+1}^{-1} \mG_{t+1}-\mD_{t}^{-1} \mG_{t})$ in Algorithm \ref{alg:PDNSGD}, we have
\begin{align}\label{eq:error-V}
\begin{split}
    &\Norm{\mV_{t+1}-\mathbf{1} \vpi^{\top} \mV_{t+1}} \\
    =&\Norm{(\mA^{K}-\mathbf{1} \vpi^\top\mA^K)(\mV_{t}+\mD_{t+1}^{-1} \mG_{t+1}-\mD_{t}^{-1} \mG_{t})} \\
    =&\Norm{(\mA^{K}-\mathbf{1} \vpi^\top)(\mV_{t}+\mD_{t+1}^{-1} \mG_{t+1}-\mD_{t}^{-1} \mG_{t})} \\
    =&\Norm{(\mA^{K}-\mathbf{1} \vpi^{\top}) \mV_t-(\mA^{K}-\mathbf{1} \vpi^{\top})\boldsymbol{1}\vpi^\top\mV_t+(\mA^{K}-\mathbf{1} \vpi^{\top}) (  \mD_{t+1}^{-1} \mG_{t+1}-\mD_{t}^{-1} \mG_{t} ))} \\
    =&\Norm{(\mA^{K}-\mathbf{1} \vpi^{\top})(\mV_t-\vone\vpi^\top\mV_t) + (\mA^{K}-\mathbf{1} \vpi^{\top}) (  \mD_{t+1}^{-1} \mG_{t+1}-\mD_{t}^{-1} \mG_{t} ))} \\
    \leq& \rho\Norm{\mV_t-\vone\vpi^\top\mV_t} + \rho\Norm{\mD_{t+1}^{-1} \mG_{t+1}-\mD_{t}^{-1} \mG_{t}},
\end{split}
\end{align}
where the second equality is based on the fact $\vpi^{\top} \mA=\vpi^{\top}$ from Proposition \ref{prop:Perron-Frobenius}; 
the third equality is based on (\ref{eq:eq1});
and the last step is based on the triangle inequality and the definition $\rho=\Norm{\mA^{K}-\mathbf{1} \vpi^{\top}}$.

We use triangle inequality to bound the term $\Norm{\mD_{t+1}^{-1} \mG_{t+1}-\mD_{t}^{-1} \mG_{t}}$ in the last line of equation (\ref{eq:error-V}) as
\begin{align}\label{eq:error-V1}
        \Norm{\mD_{t+1}^{-1} \mG_{t+1}-\mD_{t}^{-1} \mG_{t}}\leq \underbrace{\Norm{\mD_{t+1}^{-1}(\mG_{t+1}-\mG_{t})}}_{C_1}+\underbrace{\Norm{(\mD_{t+1}^{-1}-\mD_{t}^{-1}) \mG_{t}}}_{C_2}.
\end{align}
We split the term $C_1$ as 
\begin{align}\label{eq:C1}
\begin{split}
    C_1=&\underbrace{\mD_{t+1}^{-1}(\mG_{t+1}-\nabla \mF(\mX_{t+1}))}_{D_1} 
+ \underbrace{\mD_{t+1}^{-1}(\nabla \mF(\mX_{t+1}) - \nabla \mF(\mathbf{1}\vw_{t+1}))}_{D_2}  
+ \underbrace{\mD_{t+1}^{-1}(\nabla \mF(\mathbf{1}\vw_{t+1}) -\nabla \mF(\mathbf{1} \vw_{t}))}_{D_3}\\
& + \underbrace{\mD_{t+1}^{-1}(\nabla \mF(\mathbf{1} \vw_{t})-\nabla \mF(\mX_{t}))}_{D_4} 
+ \underbrace{\mD_{t+1}^{-1}(\nabla \mF(\mX_{t})-\mG_{t})}_{D_5}.
\end{split}
\end{align} 
For the term $D_1$, it holds
\begin{align}\label{ineq:grad error_t+1}
\begin{split}
    \BE[\Norm{D_1}] &=  \BE\left[\sqrt{\sum_{i=1}^n\Norm{[\mD_{t+1}^{-1}]_{ii}(\vg_{t+1,i}-\nabla f_{i}(\vx_{t+1,i}))}^2}\,\right] \\
    &\leq \BE\left[\sqrt{\sum_{i=1}^n\theta^2 \Norm{\vg_{t+1,i}-\nabla f_{i}(\vx_{t+1,i})}^2}\,\right]\\
    &\leq \BE\left[\sqrt{n\theta^2\max_{i\in[n]} \Norm{\vg_{t+1,i}-\nabla f_{i}(\vx_{t+1,i})}^2}\,\right]\\
     &= \sqrt{n}\theta \max_{i\in[n]}\BE\left[ \Norm{\vg_{t+1,i}-\nabla f_{i}(\vx_{t+1,i})}^2\right]\\
    &\leq \frac{2 \sqrt{2n} \theta \sigma}{b^{1-1 / p}},
    \end{split}
\end{align}
where the first inequality is based on Proposition \ref{lem:Bounded Diagonals} and the last step is based on  Lemma \ref{lem:heavy-tailed}.

Similarly, we bound the term $D_5$ as  
\begin{align}\label{ineq:grad error_t}
    \BE[\Norm{D_5}] \leq \frac{2 \sqrt{2n} \theta\sigma}{b^{1-1 / p}}.
\end{align}

For the term $D_2$, it holds 
\begin{align}\label{ineq:grad consensus error_t+1}
\begin{split}
    \Norm{D_2}=&\sqrt{\sum_{i=1}^n \Norm{[\mD_{t+1}^{-1}]_{ii}(\nabla f_{i}(\vx_{t+1,i})-\nabla f(\vw_{t+1}))}^2} \\
    \leq & \sqrt{\sum_{i=1}^n \theta^2 L^2\Norm{\vx_{t+1,i}-\vw_{t+1}}^2}
    =\theta L\Norm{\mX_{t+1}-\mathbf{1} \vw_{t+1}},
    \end{split}
\end{align}
where the first inequality is based on Proposition \ref{lem:Bounded Diagonals} and Assumption \ref{asm:smoothness}.

Similarly, we bound the term $D_4$ as 
\begin{align}\label{ineq:grad consensus error_t}
    \Norm{D_4}\leq \theta L\Norm{\mX_{t}-\mathbf{1} \vw_{t}}.
\end{align}

For the term $D_3$, it holds
\begin{align}\label{ineq:smothness}
\begin{split}
    \Norm{D_3}&=\sqrt{\sum_{i=1}^n \Norm{[\mD_{t+1}^{-1}]_{ii}(\nabla f_i(\vw_{t+1})-\nabla f_i(\vw_{t}))}^2} \\
    &\leq\sqrt{\theta^2 L^2\sum_{i=1}^n \Norm{\vw_{t+1}-\vw_{t}}^2}=\sqrt{n}\theta L\Norm{\vw_{t+1}-\vw_{t}}\\
    &\leq \sqrt{n}\theta L\eta\Norm{\vpi^{\top} \mU_{t}}\leq \sqrt{n}\theta L\eta,
    \end{split}
\end{align}
where first inequality is based on Assumption \ref{asm:smoothness} and the last inequality is based on equation (\ref{eq:piUt}).

Combing equations (\ref{ineq:grad error_t+1}), (\ref{ineq:grad error_t}), (\ref{ineq:grad consensus error_t+1}), (\ref{ineq:grad consensus error_t}), and (\ref{ineq:smothness}) with equation (\ref{eq:C1}), we obtain
\begin{align}\label{eq:C_1-final}
\begin{split}
        \BE[C_1] 
\leq & \frac{4\sqrt{2 n}\theta\sigma}{b^{1-1 / p}}
+ \theta L\left(\Norm{\mX_{t}-\mathbf{1} \vpi^{\top} \mX_{t}} + \Norm{\mX_{t+1}-\mathbf{1} \vpi^{\top} \mX_{t+1}}\right)
+ \sqrt{n} \theta L \eta \\
\leq & \frac{4\sqrt{2 n}\theta\sigma}{b^{1-1 / p}}
+ \theta L\left(\Norm{\mX_{t}-\mathbf{1} \vpi^{\top} \mX_{t}} + \rho(\Norm{\mX_{t}-\mathbf{1}\vpi^\top  \mX_{t}} + 2\sqrt{n} \eta)\right)
+ \sqrt{n} \theta L \eta \\
= & \frac{4\sqrt{2 n}\theta\sigma}{b^{1-1 / p}}
+ (1+\rho)\theta L\Norm{\mX_{t}-\mathbf{1} \vpi^{\top} \mX_{t}} +  (2 \rho +1)\theta L \eta \sqrt{n},
\end{split}
\end{align}
where the second inequality is based on (\ref{lem:consensus_x}).

For the term $C_2$, we have
\begin{align}\label{eq:C_2-final}
\begin{split}
    C_2 &\leq\Norm{\mD_{t+1}^{-1}-\mD_{t}^{-1}} \Norm{\mG_{t}}\\
    &\leq \Norm{\mD_{t+1}^{-1}-\mD_{t}^{-1}}(\Norm{\nabla \mF(\mX_t)}+\Norm{ \mG_{t}-\nabla \mF(\mX_t)}  )\\
    &\leq 2 \theta n^2 \kappa^{3/2}\beta^{(t+1)K} (\Norm{\nabla \mF(\mX_t)}+\Norm{ \mG_{t}-\nabla \mF(\mX_t)} )\\
    &\leq 2 \theta n^2 \kappa^{3/2}\beta^{(t+1)K} \Big(\Norm{\nabla \mF(\mX_t)} + \frac{2\sqrt{2n}\sigma}{b^{1-1 / p}}\Big)\\
    &\leq  2 \theta n^2 \kappa^{3/2}\beta^{(t+1)K}\Big(L\Norm{\mX_{t}-\mathbf{1}\vw_{t}}_{F}+\sqrt{n}\Norm{\nabla f(\vw_{t})}_{2} +  \frac{2 \sqrt{2 n}\sigma}{b^{1-1 / p}}\Big),
\end{split}
\end{align}
where the first two steps are based on triangle inequalities;
the third step is based on  Lemmas \ref{lem: iteration error for D} and \ref{lem:heavy-tailed};
and last inequality is based on the following derivation
\begin{align*}
\Norm{\nabla \mF(\mX_{t})} & =\Norm{\nabla \mF(\mX_{t})-\nabla \mF(\mathbf{1} \vw_{t})+\nabla \mF(\mathbf{1}\vw_{t})} \\
& \leq\Norm{\nabla \mF(\mX_{t})-\nabla \mF(\mathbf{1} \vw_{t})}+\Norm{\nabla \mF(\mathbf{1} \vw_{t})}\\
& =\sqrt{\sum_{i=1}^{n}\Norm{\nabla f_{i}(\vx_{t,i})-\nabla f_{i}(\vw_{t})}^{2}}+\sqrt{n}\Norm{\nabla f(\vw_{t})}\\
&\leq \sqrt{\sum_{i=1}^{n}(L\Norm{\vx_{t,i}-\vw_{t}})^{2}}+\sqrt{n}\Norm{\nabla f(\vw_{t})}\\
&=L\Norm{\mX_{t}-\mathbf{1} \vw_{t}}+\sqrt{n}\Norm{\nabla f(\vw_{t})} \\
&=L\Norm{\mX_{t}-\mathbf{1}\vpi^\top\mX_t}+\sqrt{n}\Norm{\nabla f(\vw_{t})},
\end{align*}
where the last inequality is based on Assumption \ref{asm:smoothness}.

Combing equations (\ref{eq:C_1-final}) and (\ref{eq:C_2-final}) with equations (\ref{eq:error-V}) and (\ref{eq:error-V1}), we finish the proof. 
\end{proof}

\subsection{The Proof of Theorem \ref{thm:main}}

\begin{proof}
   
For given $\mX_t$, $\mV_t$, and $\vw_t$,  the definition of Lyapunov function implies
\begin{align*}
&\BE [\Phi_{t+1}]\\
={}&\BE\left[f(\vw_{t+1}) +\frac{32\eta L}{\sqrt{n}} \Norm{\mX_{t+1}-\boldsymbol{1}\vpi^\top\mX_{t+1}}+ \frac{4\eta}{n} \Norm{\mV_{t+1}-\mathbf{1}\vpi^{\top} \mV_{t+1}}\right] \\
\leq{}& \BE\left[f(\vw_{t})-\eta \Norm{\nabla f(\vw_{t})}  
+ \frac{2\eta}{n} \Norm{n\nabla f(\vw_{t}) -\vpi^{\top} \mV_{t}} 
+ \frac{2\eta}{n}\Norm{\mV_{t}-\mathbf{1} \vpi^{\top} \mV_{t}} + \frac{\eta^{2} L}{2} \right. \\
& \left.\quad +  \frac{32\eta L}{\sqrt{n}} \Norm{\mX_{t+1}-\boldsymbol{1}\vpi^\top\mX_{t+1}}+\frac{4\eta}{n}\Norm{\mV_{t+1}-\mathbf{1}\vpi^{\top} \mV_{t+1}}\right]\\
\leq{}& f(\vw_t)-(\eta-2\eta\theta\sqrt{\kappa}\beta^{(t+1)K}-8\eta\rho\theta n^{3/2}\kappa^{3/2}\beta^{(t+1)K})\Norm{\nabla f(\vw_t)}\\
& +\left( \frac{2\eta L(1+\sqrt{\kappa}\theta \beta^{(t+1)K})}{\sqrt{n}}+\frac{32\eta L\rho}{\sqrt{n}}+\frac{4\eta\rho\theta L}{n}(1+\rho+2n^2\kappa^{3/2}\beta^{(t+1)K}) \right)\Norm{\mX_t-\boldsymbol{1}\vpi^\top \mX_t} \\
& +\left(\frac{2\eta}{n}+\frac{4\eta\rho}{n}\right)\Norm{\mV_t-\boldsymbol{1}\vpi^\top \mV_t}  + \frac{4\sqrt{2\kappa}\eta\theta\beta^{(t+1)K}\sigma}{b^{1-1/p}} \\
&+ \frac{\eta^2L}{2} + \frac{16\eta\sigma}{(nb)^{1-1/p}}  + \frac{16\sqrt{2}\eta\rho\theta\sigma(1+n^2\kappa^{3/2}\beta^{(t+1)K})}{\sqrt{n}b^{1-1/p}} + 64\eta^2L\rho+\frac{(8\rho+4)\rho\theta L\eta^2}{\sqrt{n}}\\
\leq{}& f(\vw_t)-\frac{\eta}{2}\Norm{\nabla f(\vw_t)} + \frac{32\eta L}{\sqrt{n}}\Norm{\mX_t-\boldsymbol{1}\vpi^\top \mX_t} + \frac{4\eta}{n} \Norm{\mV_t-\boldsymbol{1}\vpi^\top\mV_t}  + \frac{4\sqrt{2\kappa}\eta\theta\beta^{(t+1)K}\sigma}{b^{1-1/p}}\\
& + \frac{\eta^2L}{2} + \frac{16\eta\sigma}{(nb)^{1-1/p}}+\frac{256\eta\sigma}{\kappa nb^{1-1/p}}+64\eta^2L\rho+\frac{(8\rho+4)\rho\theta L\eta^2}{\sqrt{n}}\\
=& \Phi_t-\frac{\eta}{2}\Norm{\nabla f(\vw_t)} + \frac{4\sqrt{2\kappa}\eta\theta\beta^{(t+1)K}\sigma}{b^{1-1/p}} 
  + \frac{\eta^2L}{2} + \frac{16\eta\sigma}{(nb)^{1-1/p}}+\frac{256\eta\sigma}{\kappa nb^{1-1/p}}+64\eta^2L\rho+\frac{(8\rho+4)\rho\theta L\eta^2}{\sqrt{n}},
\end{align*}
where the first inequality is based on Lemma \ref{lem:descent-mean}; the second inequality is based on Lemmas \ref{lem:Estimate descent deviation}, \ref{lem:consensus_x}, and \ref{lem:consensus_v};
the third inequality based on the definition $\theta=2n\kappa$ and 
Lemma \ref{lem: matrixerror} with $K \geq 3(1+\ln (\kappa n))/(1-\beta)$ that lead to
\begin{align}\label{eq:upper-bound-rho}
\rho:=\Norm{\mA^{K}-\mathbf{1} \vpi^{\top}} \leq \sqrt{n\kappa}\beta^K\leq \frac{1}{n^2\kappa^2},
\end{align}
which implies the factors before the terms $\Norm{\nabla f(\vw_t)}$, $\Norm{\mX_t-\boldsymbol{1}\vpi^\top \mX_t}$ and $\Norm{\mV_t-\boldsymbol{1}\vpi^\top\mV_t}$ satisfy
\begin{align*}
\begin{cases}
\eta - 2\eta\theta  \sqrt{\kappa} \beta^{(t+1)K} - 8\eta\rho\theta n^{3/2}\kappa^{3/2}\beta^{(t+1)K} \geq \dfrac{\eta}{2}, \\[0.25cm]
\dfrac{2\eta L(1+\sqrt{\kappa}\theta \beta^{(t+1)K})}{\sqrt{n}}+\dfrac{32\eta L\rho}{\sqrt{n}}+\dfrac{4\eta\rho\theta L(1+\rho+2n^2\kappa^{3/2}\beta^{(t+1)K})}{n} \leq \dfrac{32\eta L}{\sqrt{n}}, \\[0.25cm]
\dfrac{2\eta}{n}+\dfrac{4\eta \rho}{n} \leq \dfrac{4\eta}{n}. 
\end{cases}
\end{align*}

Taking the average on above result with $t=0,1,\dots,T-1$, we have
\begin{align}\label{eq:avg-grad}
\begin{split}
    & \BE\left[\frac{1}{T}\sum_{t=0}^{T-1}\Norm{\nabla f(\vw_t)}\right] \\
\leq & \frac{2\BE[\Phi_0-\Phi_T]}{T\eta}  
+ \frac{16\sqrt{2}\sigma}{e^3n^2\kappa b^{1-1/p}} 
+ \eta L + \frac{32\sigma}{(nb)^{1-1/p}}+\frac{512\sigma}{\kappa nb^{1-1/p}} + 128\eta L\rho+\frac{(16\rho+8)\rho\theta L\eta}{\sqrt{n}} \\
\leq & \frac{2\BE[\Phi_0-\Phi_T]}{T\eta}  
+  \frac{546\sigma}{(nb)^{1-1/p}}+ 153\eta L,
\end{split}
\end{align}
where the second inequality is based on the facts $\kappa\geq 1$, $\theta=2n\kappa$, and $\rho\leq 1/(n^2\kappa^2)$ from equation (\ref{eq:upper-bound-rho}).

Now we upper bound the term $\BE[\Phi_0-\Phi_T]$ by the initial function value gap 
\begin{align*}
    \Delta = f(\bar\vx_0) - \inf _{\vx \in \BR^d} f(\vx).
\end{align*}
Specifically, we have
\begin{align}\label{eq:initial-phi}
\begin{split}
\BE[\Phi_0-\Phi_T]=&\BE\left[ f(\vw_{0}) + 
\frac{4\eta }{n}\Norm{\mV_{0}-\mathbf{1}\vpi^{\top} \mV_{0}} +  \frac{32\eta L}{\sqrt{n}} \Norm{\mX_{0}-\mathbf{1}\vpi^\top  \mX_{0}} \right.\\
&-\left.\left(f(\vw_{T}) + \frac{4\eta}{n} \Norm{\mV_{T}-\mathbf{1}\vpi^{\top} \mV_{T}} +  \frac{32\eta L}{\sqrt{n}} \Norm{\mX_{T}-\mathbf{1}\vpi^\top  \mX_{T}}\right)\right]\\
&\leq f(\bar\vx_0) - \inf _{\vx \in \BR^d} f(\vx)+\frac{4\eta}{n}\Norm{\mV_{0}-\mathbf{1}\pi^{\top} \mV_{0}} \\
&= \Delta +\frac{4\eta}{n}\Norm{\mV_{0}-\mathbf{1}\pi^{\top} \mV_{0}}.
\end{split}
\end{align}
For the last term in above equation, we have
\begin{align}\label{eq:initial-V}
\begin{split}
        &\BE\left[\Norm{\mV_{0}-\mathbf{1}\vpi^{\top} \mV_{0}}\right]=\BE\left[\Norm{\mA^{\hat{K}}\mD_0^{-1}\mG_0-\mathbf{1} \vpi^\top \mD_0^{-1}\mG_0}\right] \\
    \leq &\Norm{\mA^{\hat{K}}-\mathbf{1}\vpi^\top}\Norm{\mD_0^{-1}}\BE\left[\Norm{\mG_0}\right]\\
    \leq & 2n\kappa\sqrt{n\kappa}\beta^{\hat{K}} \BE\left[\Norm{\mG_0}\right] \\
    \leq & 2n\kappa\sqrt{n\kappa}\beta^{\hat{K}}(\Norm{\nabla f(\mX_0)} + \BE\left[\Norm{\mG_0-\nabla f(\mX_0)})\right] \\
    \leq & 2n\kappa\sqrt{n\kappa}\beta^{\hat{K}} \left(\sqrt{\sum_{i=0}^n\Norm{\nabla f_i(\bar\vx_0)}^2} + \frac{2\sqrt{2n}\sigma}{b^{1-1/p}}\right) \\
    \leq & \frac{n\Delta}{4\eta},
\end{split}
\end{align}
where the  second inequality is based on Proposition \ref{lem:Bounded Diagonals} and Lemma \ref{lem: matrixerror}, the third inequality is based on the triangle inequality, the fourth inequality is based on Lemma \ref{lem:heavy-tailed} and the last step is based on the setting 
 \begin{align*}
     \hat{K} \geq \frac{1}{1-\beta}\ln \left(1+\frac{8\kappa\eta \sqrt{n\kappa} \left(\sqrt{\sum_{i=1}^n\Norm{\nabla f_i(\bar\vx_{0})}^2}+2\sqrt{2n}\sigma/{b^{1-1/p}}\right)}{\Delta}\right) = \tilde{{\Theta}} \left(\frac{1}{1-\beta}\right).
 \end{align*}

 Combining results of equations (\ref{eq:avg-grad}),  (\ref{eq:initial-phi}), and (\ref{eq:initial-V}) and taking 
\begin{align*}
    T= \left\lceil\frac{3672L\Delta}{\epsilon^2}\right\rceil, \quad \eta=\frac{\epsilon}{918L}, \quad \text{and}  \quad b= \left\lceil \frac{1}{n} \left(\frac{3276\sigma}{\epsilon}\right)^{\frac{p}{p-1}}\right\rceil,
\end{align*}
we obtain
\begin{align*}
    \BE\left[\frac{1}{T}\sum_{t=0}^{T-1}\Norm{\nabla f(\vw_t)}\right]
\leq \frac{\epsilon}{6} + \frac{\epsilon}{6} + \frac{\epsilon}{6} \leq \frac{\epsilon}{2}.
\end{align*}

Recall that each agent sample $\hat\vx_i\sim{\rm Unif}\{\vx_{0,i},\dots,\vx_{T-1,i}\}$. 
Hence, for all $i\in[n]$, we have
\begin{align*}
     \BE [\Norm{\nabla f(\hat{\vx}_i)}] &= \BE\left[ \frac{1}{T}\sum_{t=0}^{T-1}\Norm{\nabla f(\vx_{t,i})}\right]\\
     &\leq \BE\left[ \frac{1}{T}\sum_{t=0}^{T-1}\Norm{\nabla f(\vx_{t,i})-\nabla f(\vw_t)}+\frac{1}{T}\sum_{t=0}^{T-1}\Norm{\nabla f(\vw_t)}\right]\\
     &\leq \BE \left[ \frac{1}{T}\sum_{t=0}^{T-1}L\Norm{\vx_{t,i}-\vw_t} +
      \frac{1}{T}\sum_{t=0}^{T-1}\Norm{\nabla f(\vw_t)}\right]\\
     &\leq \BE \left[ \frac{1}{T}\sum_{t=0}^{T-1}L\Norm{\mX_t-\boldsymbol{1}\vw_t}\right] + \frac{\epsilon}{2}\\
     &\leq \frac{L}{T}\cdot\frac{2n\eta\rho T}{1-\rho}+\frac{\epsilon}{2}\\
     &=\frac{2n\eta\rho L}{1-\rho}+ \frac{\epsilon}{2},
 \end{align*}
where the first inequality is based on the triangle inequality; the second inequality is based on Assumption \ref{asm:smoothness}; 
the third inequality is based on Theorem \ref{thm:main};
the last inequality applies Lemma \ref{lem:consensus_x} that implies
\begin{align*}
\begin{split}
    &\Norm{\mX_{t+1} - \boldsymbol{1}\vpi^\top \mX_{t}}\\
    \leq &\rho(\Norm{\mX_{t} - \boldsymbol{1}\vpi^\top \mX_{t}} +2n\eta)\\
    \leq & \rho(\rho(\Norm{\mX_{t-1} - \boldsymbol{1}\vpi^\top \mX_{t-1}} +2n\eta)+2n\eta) \\
    =&\rho^2 \Norm{\mX_{t-1} - \boldsymbol{1}\vpi^\top \mX_{t-1}}+2n\eta(\rho+\rho^2) \\
    \leq&\rho^{t+1} \Norm{\mX_0-\boldsymbol{1}\vpi^\top\mX_0}+2n\eta \sum_{i=1}^{t+1} \rho^i \\
    =&2n\eta\sum_{i=1}^{t+1} \rho^i \\
    \leq & \frac{2n\eta\rho}{1-\rho}.
\end{split}
\end{align*}

In order to achieve $\BE [\Norm{\nabla f(\hat{\vx_i})}]  \leq \epsilon$ for all $i\in[n]$, we require $2n\eta\rho L/(1-\rho) \leq \epsilon/2$, which can be achieved by taking 
\begin{align*}
     K  \geq \max\left\{\frac{3(1+\ln (\kappa n))}{1-\beta},
    \frac{\ln {\sqrt{n\kappa}(\epsilon+4n\eta L)}/{\epsilon}}{1-\beta}\right\}
\end{align*}
that leads to 
\begin{align*}
    \rho \leq \frac{\epsilon}{\epsilon+4n\eta L}.
\end{align*}
Additionally, the parameter settings requires the overall sample complexity of
\begin{align*}
Tnb=\mathcal{O}\left(\frac{\Delta L \sigma^{\frac{p}{p-1}}}{\epsilon^{\frac{3p-2}{p-1}}}\right)
\end{align*}
and the communication complexity of
\begin{align*}
\hat K + TK 
=\mathcal{O}\left(\frac{\Delta L}{(1-\beta)\epsilon^{2}}\right).
\end{align*}
\end{proof}

\section{The Proof of Results in Section \ref{sec:undirected}}\label{appendix:undirected}

For the undirected network, we adapt Algorithm \ref{alg:PDNSGD} by modifying the update rules of the variables $\mX_{t+1}$ and $\mV_{t+1}$ to
\begin{align*} 
\begin{cases}
\mX_{t+1} = \AG\left(\mX_{t} - \eta\mU_t, \mA, K\right), \\
\mV_{t+1} = \AG\left(\mV_{t} + \mG_{t+1} - \mG_{t}, \mA, K\right).    
\end{cases}
\end{align*}
Similarly, we also apply Chebyshev acceleration to set 
\begin{align*} 
\mV_0 = \AG\big(\mG_0, \mA, \hat K\big),
\end{align*}
which leads to the algorithm reduces to normalized stochastic gradient descent with gradient tracking (DNSGD) \citep{luo2025decentralized}.
We present the complete procedure DNSGD in Algorithm \ref{alg:DNSGD}.
Different with \citet{luo2025decentralized} analyzing the relaxed smoothness under bounded variance and bounded gradient dissimilarity, this section focus on addressing the heavy-tailed noise by DNSGD.

\begin{algorithm*}[h]
\caption{Normalized Stochastic Gradient Descent with  Gradient Tracking (DNSGD)}
\label{alg:DNSGD}
\begin{algorithmic}[1]
\STATE \textbf{Input:} $\bx_0\in\BR^{1\times d}$, $\mA\in\BR^{n\times n}$, $\eta>0$, $\hat K,K,T,b\in\BN$. \\[0.12cm]
        \STATE $\mX_0=\vone\bx_0$;\quad independently sample $\vxi_{0,i,k}\sim\fD_i$ for all $k\in[b]$ at each agent $i$;  \\[0.12cm]
        \STATE \label{line:G02} $\mG_0 = \begin{bmatrix}
              {\displaystyle\frac{1}{b}\sum_{k=1}^{b} \nabla F_1(\bx_0;\vxi_{0,1,k})} & \dots & 
              {\displaystyle\frac{1}{b}\sum_{k=1}^{b} \nabla F_n(\bx_0;\vxi_{0,n,k})}
            \end{bmatrix}^\top$;\quad $\mV_0 = \AG\big(\mG_0, \mA, \hat K\big)$; \\[0.12cm]
        \STATE \textbf{for} $t = 0, 1, \dots, {T-1}$ \textbf{do} \\[0.12cm]   		     
        \STATE\quad\label{line:U2}$\mU_t = \begin{bmatrix}
            \dfrac{\vv_{t,1}}{\Norm{\vv_{t,1}}} & \dots & \dfrac{\vv_{t,n}}{\Norm{\vv_{t,n}}}
        \end{bmatrix}^\top$; \quad $\mX_{t+1} = \AG\left(\mX_{t} - \eta\mU_t, \mA, K\right)$ ; \label{line:update-X_undirected2} \\[0.12cm]
        \STATE\quad independently sample $\vxi_{t+1,i,k}\sim\fD_i$ for all $k\in[b]$ at each agent $i$; \\[0.12cm]
        \STATE\quad \label{line:Gt2} $\displaystyle{\mG_{t+1} =\begin{bmatrix}
        {\displaystyle\frac{1}{b} \sum_{k=1}^{b} \nabla F_{1}(\vx_{t+1,1}; \vxi_{t+1,1,k})} 
        & \cdots &
        {\displaystyle\frac{1}{b} \sum_{k=1}^{b} \nabla F_{n}(\vx_{t+1,n}; \vxi_{t+1,n,k})}\end{bmatrix}^\top}$; \\[0.15cm]
        \STATE\quad \label{line:Vt2} $\mV_{t+1} = \AG\left(\mV_{t} + \mG_{t+1} - \mG_{t}, \mA, K\right)$; \\[0.2cm]       
        \STATE\textbf{end for} \\[0.12cm]
        \STATE $\hat\vx_i\sim{\rm Unif}\{\vx_{0,i},\dots,\vx_{T-1,i}\}$ for each agent $i$. \\[0.12cm] 
		\STATE \textbf{Output:} $\hat\vx_i$ for each agent $i$.
\end{algorithmic}
\end{algorithm*}

\subsection{The proof of Lemma \ref{lem:descent-undirected}}
\begin{proof}
Recall that $\bar{\vx}_{t} = \frac{1}{n}\boldsymbol{1}^{\top} \mX_{t}$, then the update $\mX_{t+1} = \AG\left(\mX_{t} - \eta\mU_t, \mA, K\right)$ and Proposition \ref{prop:Chebyshev acceleration} imply
\begin{align*}
\bar{\vx}_{t+1}=\bar{\vx_{t}}-\frac{\eta}{n}\boldsymbol{1}^{\top}  \mU_{t}.     
\end{align*}
According to above equation and Assumption \ref{asm:smoothness}, we have
\begin{align}\label{neq:smoothness_undirected}  
    f(\bar{\vx}_{t+1}) \leq f(\bar{\vx}_t)- \frac{\eta}{n}\left\langle \nabla f(\bar{\vx}_{t}), \boldsymbol{1}^{\top} \mU_{t}\right\rangle+\frac{\eta^{2} L}{2n^2}\Norm{\boldsymbol{1}^{\top} \mU_{t}}^{2}.
\end{align}
We then focus on the inner product term on the right-hand side of (\ref{neq:smoothness}) as follows
\begin{align}\label{neq:inner_undirected} 
\begin{split}    
&- \frac{\eta}{n}\left\langle \nabla f(\bar{\vx}_{t}), \boldsymbol{1}^{\top} \mU_{t}\right\rangle \\
&=-\frac{\eta}{n}\left\langle \nabla f\left(\bar{\vx}_{t}\right), \sum_{i=1}^n \frac{\vv_{t,i}}{\Norm{\vv_{t,i}}}\right\rangle \\
&=-\frac{\eta}{n} \left(\sum_{i=1}^n \frac{\nabla f(\bar{\vx}_{t})^\top\vv_{t,i}}{\Norm{\vv_{t,i}}}\right)\\
&\leq  -\frac{\eta}{n}\left(\sum_{i=1}^n (\Norm{\nabla f(\bar{\vx}_{t})}-2\Norm{\nabla f(\bar{\vx}_{t})-\vv_{t,i}})\right)\\
&=-\frac{\eta}{n} \left( n\Norm{\nabla f(\bar{\vx}_{t})}-2 \sum_{i=1} ^n\Norm{\nabla f(\bar{\vx}_{t})-\vv_{t,i}} \right)\\
&\leq -\eta \Norm{\nabla f(\bar{\vx}_{t})}  + \frac{2\eta}{n} \left( \sum_{i=1}^n \Norm{\nabla f(\bar{\vx}_{t})-\frac{1}{n}\boldsymbol{1}^{\top} \mV_{t}}+\sum_{i=1}^n \Norm{\frac{1}{n}\boldsymbol{1}^{\top} \mV_{t}-\vv_{t,i}} \right)\\
& \leq -\eta \Norm{\nabla f(\bar{\vx}_{t})}  + 2\eta\Norm{\nabla f(\bar{\vx}_{t})-\frac{1}{n}\boldsymbol{1}^{\top} \mV_{t}}+ \frac{2\eta}{n}\sqrt{n}\Norm{\mV_{t}-\frac{1}{n}\boldsymbol{1} \boldsymbol{1}^{\top} \mV_{t}} \\
&= -\eta \Norm{\nabla f(\bar{\vx}_{t})}  + 2\eta\Norm{\nabla f(\bar{\vx}_{t})-\frac{1}{n}\boldsymbol{1}^{\top} \mV_{t}}+ \frac{2\eta}{\sqrt{n}}\Norm{\mV_{t}-\frac{1}{n}\boldsymbol{1} \boldsymbol{1}^{\top} \mV_{t}} 
\end{split}
\end{align}
where the first inequality is based on Lemma \ref{lem: baseineq}; 
the second inequality is based on triangle inequality; 
the last step uses Cauchy--Schwarz inequality which leads to 
\begin{align*}
    \sum_{i=1}^n \Norm{\frac{1}{n}\boldsymbol{1}^{\top} \mV_{t}-\vv_{t,i}}\leq \sqrt{\sum_{i=1}^n1^2}\sqrt{\sum_{i=1}^n \Norm{\frac{1}{n}\boldsymbol{1}^\top\mV_t-\vv_{t,i}}^2} 
    \leq \sqrt{n}\Norm{\frac{1}{n}\boldsymbol{1} \boldsymbol{1}^{\top}  \mV_{t}-\mV_{t}}.
\end{align*}
Combing equations (\ref{neq:smoothness_undirected}) and (\ref{neq:inner_undirected}), we obtain 
\begin{align*}
\begin{split}    
     f(\bar{\vx}_{t+1})
    & \leq f(\bar{\vx}_{t})-\eta \Norm{\nabla f(\bar{\vx}_{t})}  + 2\eta  \Norm{\nabla f(\bar{\vx}_{t})-\frac{1}{n}\boldsymbol{1}^{\top} \mV_{t}}+ \frac{2\eta}{\sqrt{n}}\Norm{\mV_{t}-\frac{1}{n}\boldsymbol{1}\boldsymbol{1}^{\top} \mV_{t}}   +\frac{\eta^{2} L}{2n^2}\Norm{\boldsymbol{1}^{\top} \mU_{t}}^{2}\\
    &\leq f(\bar{\vx}_{t})-\eta \Norm{\nabla f(\bar{\vx}_{t})}  + 2\eta  \Norm{\nabla f(\bar{\vx}_{t})-\frac{1}{n}\boldsymbol{1}^{\top} \mV_{t}}+ \frac{2\eta}{\sqrt{n}}\Norm{\mV_{t}-\frac{1}{n}\boldsymbol{1}\boldsymbol{1}^{\top} \mV_{t}} +\frac{\eta^{2} L}{2},
\end{split}
\end{align*}
where the last step is based on facts that the norm of each row of $\mU_t$ is one .
Taking expectations on both sides of the above inequality, then we finish the proof.

\end{proof}

\subsection{The proof of Lemma \ref{lem:Estimate descent deviation undirected}}
\begin{proof}
Similar to the proof of Lemma \ref{lem:Estimate descent deviation}
, we have 
\begin{align}\label{eq:piVDG_undirected}
 \frac{1}{n}\boldsymbol{1}^{\top}\mV_t= \frac{1}{n}\boldsymbol{1}^\top\mG_{t}.
\end{align}

We now consider the statement of this lemma. According equation (\ref{eq:piVDG_undirected}) and the triangle inequality, we split the upper bound of $\Norm{\nabla f(\bar{\vx}_{t})-\frac{1}{n}\boldsymbol{1}^{\top} \mV_t}$ as follows
\begin{align}\label{eq:diff-grad-V_undirected}
\begin{split}
 \Norm{\nabla f(\bar{\vx}_{t})-\frac{1}{n}\boldsymbol{1}^{\top} \mV_t} = \Norm{\nabla f(\bar{\vx}_{t})-\frac{1}{n}\boldsymbol{1}^{\top} \mG_t}  
\leq  \underbrace{\Norm{\nabla f(\bar{\vx}_{t})-\frac{1}{n}\mathbf{1}^{\top} \nabla \mF(\mX_{t})}}_{A_{1}} 
+\underbrace{\Norm{\frac{1}{n}\boldsymbol{1^\top}(\nabla \mF(\mX_{t})-\mG_{t})}}_{A_2} 
\end{split}
\end{align}
For the term $A_1$, it holds
\begin{align}\label{eq:A_1_undirected}
\begin{split}
     A_{1}&=\Norm{\nabla f(\bar{\vx}_{t})-\frac{1}{n}\mathbf{1}^{\top} \nabla \mF(\mX_{t})}= \frac{1}{n}\Norm{\sum_{i=1}^{n}(\nabla f_{i}(\bar{\vx}_{t})-\nabla f_{i}(\vx_{t,i}))}\\
    &\leq \frac{L}{n}\sum_{i=1}^{n}\Norm{\bar{\vx}_{t}-\vx_{t,i}}
     \leq \frac{L}{n}\sqrt{n}\sqrt{\sum_{i=1}^{n}\Norm{\bar{\vx}_{t}-\vx_{t,i}}^{2}} =\frac{L}{\sqrt{n}}\Norm{\mathbf{1} \bar{\vx}_{t}-\mX_{t}} = \frac{L}{\sqrt{n}}\Norm{\mX_{t}-\frac{1}{n}\mathbf{1}\boldsymbol{1}^\top \mX_{t}},
\end{split}
\end{align}
where the fist inequality is based on the Assumption \ref{asm:smoothness} and the last inequality is based on the Cauchy–Schwarz inequality.

We then bound the term $A_2$ as
\begin{align}\label{eq:A_2_undirected}
\begin{split}
    \BE[A_2] &= \frac{1}{n}\BE\left[\Norm{\sum_{i=1}^n \left( \nabla f_i(\vx_{t,i}) - \vg_{t,i}\right)}\right]  \\
    &=\frac{1}{n}\BE\left[\Norm{\sum_{i=1}^n \left(\nabla f_i(\vx_{t,i}) - \frac{1}{b}\sum_{k=1}^b\nabla F_i(\vx_{t,i};\xi_{t,i,k})\right)}\right]  \\
    &\leq \frac{2\sqrt{2}}{n}\BE \left[\left(\sum_{i=1}^n\Norm{\nabla f_i(\vx_{t,i}) - \frac{1}{b}\sum_{k=1}^b\nabla F_i(\vx_{t,i};\xi_{t,i,k})}^p\right)^\frac{1}{p}\right]\\
   &\leq \frac{2\sqrt{2}}{n} \left(\BE\left[\sum_{i=1}^n\Norm{\nabla f_i(\vx_{t,i}) - \frac{1}{b}\sum_{k=1}^b\nabla F_i(\vx_{t,i};\xi_{t,i,k})}^p\right]\right)^\frac{1}{p} \\
   &\leq \frac{2\sqrt{2}}{n}\left(n \left(\frac{2\sqrt{2}\sigma}{b^{1-1/p}}\right)^p\right)^{\frac{1}{p}} = \frac{8\sigma}{(nb)^{1-1/p}}
\end{split}
\end{align}
where the first inequality is based on Lemma \ref{lem:heavy-tailed_cite};
the second inequality is based on the Jensen’s inequality;
and the last inequality is based on Lemma \ref{lem:heavy-tailed}.

Combining equations (\ref{eq:A_1_undirected}) and (\ref{eq:A_2_undirected}) with equation (\ref{eq:diff-grad-V_undirected}), we  completes the proof.
\end{proof}

\subsection{The proof of Lemma \ref{lem:consensus_x_undirected}}
\begin{proof}
According to the update $\mX_{t+1}=\AG(\mX_{t}-\eta\mU_t, \mA, K)$ in Algorithm \ref{alg:DNSGD}, we have
\begin{align*}
    &\Norm{\mX_{t+1}-\frac{1}{n}\mathbf{1} \mathbf{1} ^{\top} \mX_{t+1}}\\
    \leq & \tilde\rho \Norm{\mX_t-\eta\mU_t-\frac{1}{n}\boldsymbol{1}\boldsymbol{1}^\top(\mX_{t}-\eta\mU_t)}\\
    =&\tilde\rho \Norm{\mX_t-\eta\mU_t-\frac{1}{n}\boldsymbol{1}\boldsymbol{1}^\top\mX_t+\frac{\eta}{n}\boldsymbol{1}\boldsymbol{1}^\top\mU_t}\\
    \leq & \tilde\rho \Norm{\mX_t-\frac{1}{n}\boldsymbol{1}\boldsymbol{1}^\top\mX_t}+\rho \eta\Norm{\mU_t-\frac{1}{n}\boldsymbol{1}\boldsymbol{1}^\top\mU_t}\\
    \leq& \tilde\rho\Norm{\mX_t-\frac{1}{n}\mathbf{1} \mathbf{1} ^{\top}\mX_t} + \rho\eta\Norm{\mU_t}\\
    =&\tilde\rho\Norm{\mX_t-\frac{1}{n}\mathbf{1} \mathbf{1} ^{\top}\mX_t}+\tilde\rho\eta\sqrt{n}
\end{align*}
where the first inequality is based on Proposition \ref{prop:Chebyshev acceleration} with $\tilde\rho = c_1(1-c_2\sqrt{1-\beta})^K$, the second inequality is based on triangle inequality and the last inequality is based on Lemma \ref{lem:undirected_norm_ineq}.

\end{proof}

\subsection{The proof of Lemma \ref{lem:consensus_v_undirected}}
\begin{proof}
According to the update $\mV_{t+1}=\AG(\mV_{t}+ \mG_{t+1}- \mG_{t}, \mA, K)$ in Algorithm \ref{alg:DNSGD}, we have
\begin{align}\label{eq:error-V_undirected}
\begin{split}
    &\Norm{\mV_{t+1}-\frac{1}{n}\mathbf{1} \mathbf{1} ^{\top} \mV_{t+1}}\\
    \leq & \tilde\rho \Norm{\mV_t+
    \mG_{t+1}-\mG_t-\frac{1}{n}\boldsymbol{1}\boldsymbol{1}^\top(\mV_t+\mG_{t+1}-\mG_t)}\\
    =&\tilde\rho \Norm{\mV_t-\frac{1}{n}\boldsymbol{1}\boldsymbol{1}^\top\mV_t+(\mI-\frac{1}{n}\boldsymbol{1}\boldsymbol{1}^\top)(\mG_{t+1}-\mG_t)}\\
    \leq & \tilde\rho \Norm{\mV_t-\frac{1}{n}\boldsymbol{1}\boldsymbol{1}^\top\mV_t}+\tilde\rho \Norm{\mG_{t+1}-\mG_t},
    \end{split}
\end{align}
where the first inequality is based on Proposition \ref{prop:Chebyshev acceleration} with $\tilde\rho = c_1(1-c_2\sqrt{1-\beta})^K$, the second inequality is based on triangle inequality and the last inequality is based on Lemma \ref{lem:undirected_norm_ineq}.

We split the  term $\Norm{ \mG_{t+1}- \mG_{t}}$ in the last line of equation (\ref{eq:error-V_undirected}) as
\begin{align}\label{eq:C1_undirected}
\begin{split}
    \Norm{ \mG_{t+1}- \mG_{t}}=&\underbrace{\mG_{t+1}-\nabla \mF(\mX_{t+1})}_{D_1} 
+ \underbrace{\nabla \mF(\mX_{t+1}) - \nabla \mF(\mathbf{1}\vw_{t+1})}_{D_2}  
+ \underbrace{\nabla \mF(\mathbf{1}\vw_{t+1}) -\nabla \mF(\mathbf{1} \vw_{t})}_{D_3}\\
& + \underbrace{\nabla \mF(\mathbf{1} \vw_{t})-\nabla \mF(\mX_{t})}_{D_4} 
+ \underbrace{\nabla \mF(\mX_{t})-\mG_{t}}_{D_5}.
\end{split}
\end{align} 
For the term $D_1$, it holds
\begin{align}\label{ineq:grad error_t+1_undirected}
\begin{split}
    \BE[\Norm{D_1}] &=  \BE\left[\sqrt{\sum_{i=1}^n\Norm{[\vg_{t+1,i}-\nabla f_{i}(\vx_{t+1,i})}^2}\,\right] \\
    &\leq \BE\left[\sqrt{n\max_{i\in[n]} \Norm{\vg_{t+1,i}-\nabla f_{i}(\vx_{t+1,i})}^2}\,\right]\\
     &= \sqrt{n}\max_{i\in[n]}\BE\left[ \Norm{\vg_{t+1,i}-\nabla f_{i}(\vx_{t+1,i})}^2\right]\\
    &\leq \frac{2 \sqrt{2n} \sigma}{b^{1-1 / p}},
    \end{split}
\end{align}
where the last step is based on  Lemma \ref{lem:heavy-tailed}.

Similarly, we bound the term $D_5$ as  
\begin{align}\label{ineq:grad error_t_undirected}
    \BE[\Norm{D_5}] \leq \frac{2 \sqrt{2n} \sigma}{b^{1-1 / p}}.
\end{align}

For the term $D_2$, it holds 
\begin{align}\label{ineq:grad consensus error_t+1_undirected}
\begin{split}
    \Norm{D_2}=&\sqrt{\sum_{i=1}^n \Norm{\nabla f_{i}(\vx_{t+1,i})-\nabla f(\bar{\vx}_{t+1})}^2} \\
    \leq & L\sqrt{\sum_{i=1}^n \Norm{\vx_{t+1,i}-\bar{\vx}_{t+1}}^2}
    =L\Norm{\mX_{t+1}-\mathbf{1} \bar{\vx}_{t+1}},
    \end{split}
\end{align}
where the first inequality is based on  Assumption \ref{asm:smoothness}.

Similarly, we bound the term $D_4$ as 
\begin{align}\label{ineq:grad consensus error_t_undirected}
    \Norm{D_4}\leq L\Norm{\mX_{t}-\mathbf{1} \bar{\vx}_{t}}.
\end{align}

For the term $D_3$, it holds
\begin{align}\label{ineq:smothness_undirected}
\begin{split}
    \Norm{D_3}&=\sqrt{\sum_{i=1}^n \Norm{\nabla f_i(\bar{\vx}_{t+1})-\nabla f_i(\bar{\vx}_{t})}^2} \\
    &\leq\sqrt{L^2\sum_{i=1}^n \Norm{\bar{\vx}_{t+1}-\bar{\vx}_{t}}^2}=\sqrt{n} L\Norm{\bar{\vx}_{t+1}-\bar{\vx}_{t}}\\
    &\leq \sqrt{n} L\eta\Norm{\frac{1}{n}\boldsymbol{1}^{\top} \mU_{t}}=\sqrt{n} L\eta,
    \end{split}
\end{align}
where first inequality is based on Assumption \ref{asm:smoothness} .

Combing equations (\ref{ineq:grad error_t+1_undirected}), (\ref{ineq:grad error_t_undirected}), (\ref{ineq:grad consensus error_t+1_undirected}), (\ref{ineq:grad consensus error_t_undirected}), and (\ref{ineq:smothness_undirected}) with equation (\ref{eq:C1_undirected}), we obtain
\begin{align}\label{eq:C_1-final_undirected}
\begin{split}
        \BE[\Norm{ \mG_{t+1}- \mG_{t}}] 
\leq & \frac{4\sqrt{2 n}\sigma}{b^{1-1 / p}}
+ L\left(\Norm{\mX_{t}-\frac{1}{n}\mathbf{1} \mathbf{1}^{\top} \mX_{t}} + \Norm{\mX_{t+1}-\frac{1}{n}\mathbf{1} \mathbf{1}^{\top}  \mX_{t+1}}\right)
+ \sqrt{n} L \eta \\
\leq & \frac{4\sqrt{2 n}\sigma}{b^{1-1 / p}}
+ L\left(\Norm{\mX_{t}-\frac{1}{n}\mathbf{1} \mathbf{1}^{\top} \mX_{t}} + \tilde\rho(\Norm{\mX_{t}-\frac{1}{n}\mathbf{1}\mathbf{1}^\top  \mX_{t}} + 2\sqrt{n} \eta)\right)
+ \sqrt{n} L \eta \\
= & \frac{4\sqrt{2 n}\sigma}{b^{1-1 / p}}
+ (1+\tilde\rho) L\Norm{\mX_{t}-\frac{1}{n}\mathbf{1} \mathbf{1}^{\top} \mX_{t}} +  (2 \tilde\rho +1) L \eta \sqrt{n},
\end{split}
\end{align}
where the second inequality is based on Lemma (\ref{lem:consensus_x}) with $\vpi=\frac{1}{n}\boldsymbol{1}$.

Combing equations (\ref{eq:C_1-final_undirected}) with equations (\ref{eq:error-V_undirected}), we finish the proof. 
\end{proof}

\subsection{The proof of Theorem \ref{thm:main2}}
\begin{proof}
For the undirected network, we define the Lyapunov function
\begin{align*}
    \Psi_t := &f(\bar{\vx}_{t})  +  \frac{32\eta L}{\sqrt{n}}\Norm{\mX_{t}-\frac{1}{n}\mathbf{1}\mathbf{1}^\top  \mX_{t}} + \frac{4\eta}{\sqrt{n}} \Norm{\mV_{t}-\frac{1}{n}\mathbf{1}\mathbf{1}^{\top} \mV_{t}}.
\end{align*}
Then it holds
\begin{align*}
&\BE [\Psi_{t+1}]\\
=&\BE\left[f(\bar{\vx}_{t+1}) +\frac{32\eta L}{\sqrt{n}} \Norm{\mX_{t+1}-\frac{1}{n}\boldsymbol{1}\boldsymbol{1}^\top\mX_{t+1}}+ \frac{4\eta}{\sqrt{n}} \Norm{\mV_{t+1}-\frac{1}{n}\boldsymbol{1}\boldsymbol{1}^\top \mV_{t+1}}\right] \\
\leq & f(\bar{\vx}_{t})-\eta \Norm{\nabla f(\bar{\vx}_{t})} + \frac{\eta^{2} L}{2}+ 2\eta  \Norm{\nabla f(\bar{\vx}_{t}) -\frac{1}{n}\boldsymbol{1}^\top \mV_{t}}+ \frac{2\eta}{\sqrt{n}}\Norm{\mV_{t}-\frac{1}{n}\boldsymbol{1}\boldsymbol{1}^\top  \mV_{t}}\\
&+\left[ \frac{32\eta L}{\sqrt{n}} \Norm{\mX_{t+1}-\frac{1}{n}\boldsymbol{1}\boldsymbol{1}^\top\mX_{t+1}}+ \frac{4\eta}{\sqrt{n}} \Norm{\mV_{t+1}-\frac{1}{n}\boldsymbol{1}\boldsymbol{1}^\top \mV_{t+1}} \right] \\
\leq & f(\bar{\vx}_{t})-\eta \Norm{\nabla f(\bar{\vx}_{t})} + \frac{\eta^{2} L}{2}+2\eta \left(\frac{L}{\sqrt{n}}\Norm{\mX_t-\frac{1}{n}\boldsymbol{1}^\top\mX_t}+\frac{8\sigma}{(nb)^{1-1/p}}\right)\\
&+\frac{2\eta}{\sqrt{n}}\Norm{\mV_{t}-\frac{1}{n}\boldsymbol{1}\boldsymbol{1}^\top  \mV_{t}}+\frac{32\eta L}{\sqrt{n}}\tilde\rho\left(\Norm{\mX_t-\frac{1}{n}\boldsymbol{1}\boldsymbol{1}^\top\mX_t} +\sqrt{n}\eta \right)\\
&+\frac{4\eta}{\sqrt{n}}  \tilde\rho\Big(\Norm{\mV_{t}-\frac{1}{n}\boldsymbol{1}\boldsymbol{1}^\top\mV_{t}}+L(1+\tilde\rho)\Norm{\mX_{t}-\frac{1}{n}\boldsymbol{1}\boldsymbol{1}^\top \mX_{t}}+  (2 \tilde\rho +1) L \eta \sqrt{n} + \frac{4 \sqrt{2 n}\sigma}{b^{1-1 / p}}\Big)\\
= & f(\bar{\vx}_{t})-\eta \Norm{\nabla f(\bar{\vx}_{t})} + \frac{\eta^{2} L}{2}+(\frac{2\eta L}{\sqrt{n}}+\frac{32\eta L\tilde\rho}{\sqrt{n}}+\frac{4\eta L\tilde\rho}{\sqrt{n}}+\frac{4\eta L\tilde\rho^2}{\sqrt{n}})\Norm{\mX_t-\frac{1}{n}\boldsymbol{1}\boldsymbol{1}^\top \mX_t}\\
&+(\frac{2\eta}{\sqrt{n}}+\frac{4\eta\tilde\rho}{\sqrt{n}})\Norm{\mV_t-\frac{1}{n}\boldsymbol{1}\boldsymbol{1}^\top\mV_t}+\frac{16\eta\sigma}{(nb)^{1-1/p}}+32\eta^2 L\tilde\rho +\frac{16\sqrt{2}\eta\tilde\rho\sigma}{b^{1-1/p}}+(8\tilde\rho^2 +4\tilde\rho)L\eta^2 \\
\leq & f(\bar{\vx}_{t})-\eta \Norm{\nabla f(\bar{\vx}_{t})} + \frac{\eta^{2} L}{2} +\frac{32\eta L}{\sqrt{n}}\Norm{\mX_t-\frac{1}{n}\boldsymbol{1}\boldsymbol{1}^\top \mX_t}+\frac{4\eta}{\sqrt{n}}\Norm{\mV_t-\frac{1}{n}\boldsymbol{1}\boldsymbol{1}^\top\mV_t}\\
&+\frac{16\eta\sigma}{(nb)^{1-1/p}}+32\eta^2 L\tilde\rho +\frac{16\sqrt{2}\eta\sigma}{(nb)^{1-1/p}}+(8\tilde\rho^2 +4\tilde\rho)L\eta^2\\
\leq &\Psi_t-\eta \Norm{\nabla f(\bar{\vx}_{t})} + \frac{\eta^{2} L}{2} +\frac{32\sqrt{2}\eta\sigma}{(nb)^{1-1/p}}+44\eta^2 L\tilde\rho 
\end{align*}
where the first inequality is based on Lemma \ref{lem:descent-undirected}; the second inequality is based on Lemmas \ref{lem:Estimate descent deviation undirected}, \ref{lem:consensus_x_undirected}, and \ref{lem:consensus_v_undirected};
the third inequality based on the Proposition \ref{prop:Chebyshev acceleration} with $K = \mathcal{O}\big((1-\beta)^{-\frac{1}{2}}{\operatorname{ln}n}\big)$ that lead to
\begin{align}\label{eq:upper-bound-rho_undirected}
\tilde\rho:=c_1\left(1-c_2\sqrt{1-\beta}\,\right)^K \leq \operatorname{min} \left\{\frac{1}{2},\frac{1}{n^{1-1/p}}\right\},
\end{align}
which implies the factors before the terms $\Norm{\mX_t-\frac{1}{n}\boldsymbol{1}\boldsymbol{1}^\top \mX_t}$ and $\Norm{\mV_t-\frac{1}{n}\boldsymbol{1}\boldsymbol{1}^\top\mV_t}$ satisfy
\begin{align*}
\begin{cases}
\dfrac{2\eta L}{\sqrt{n}}+\dfrac{32\eta L\tilde\rho}{\sqrt{n}}+\dfrac{4\eta L(\tilde\rho+\tilde\rho^2)}{\sqrt{n}} \leq \dfrac{32\eta L}{\sqrt{n}}, \\[0.3cm]
\dfrac{2\eta}{\sqrt{n}}+\dfrac{4\eta \tilde\rho}{\sqrt{n}} \leq \dfrac{4\eta}{\sqrt{n}}.
\end{cases}
\end{align*}
and 
\begin{align*}
    \frac{16\sqrt{2}\eta\tilde\rho\sigma}{b^{1-1/p}}\leq \frac{16\sqrt{2}\eta\sigma}{(nb)^{1-1/p}}.
\end{align*}

Taking the average on above result with $t=0,1,\dots,T-1$, we have
\begin{align*}
    \frac{1}{T}\sum _{t=0}^{T-1}\Norm{\nabla f(\bar{\vx}_t)} &\leq \frac{\Delta_\Psi}{T\eta}+\frac{\eta L}{2} +\frac{32\sqrt{2}\sigma}{(nb)^{1-1/p}}+44\eta L\tilde\rho \leq \frac{\Delta_\Psi}{T\eta}+\frac{32\sqrt{2}\sigma}{(nb)^{1-1/p}}+45\eta L
\end{align*}

Now we upper bound the term $\BE[\Psi_0-\Psi_T]$ by the initial function value gap 
\begin{align*}
    \Delta = f(\bar\vx_0) - \inf _{\vx \in \BR^d} f(\vx).
\end{align*}
Specifically, we have
\begin{align*}
\BE[\Psi_0-\Psi_T]=&\BE\left[ f(\bar\vx_0) + 
\frac{4\eta }{\sqrt{n}}\Norm{\mV_{0}-\frac{1}{n}\mathbf{1}\mathbf{1}^{\top} \mV_{0}} +  \frac{32\eta L}{\sqrt{n}} \Norm{\mX_{0}-\frac{1}{n}\mathbf{1}\mathbf{1}^{\top}  \mX_{0}} \right.\\
&-\left.\left(f(\bar\vx_T) + \frac{4\eta}{\sqrt{n}} \Norm{\mV_{T}-\frac{1}{n}\mathbf{1}\mathbf{1}^{\top}\mV_{T}} +  \frac{32\eta L}{\sqrt{n}} \Norm{\mX_{T}-\frac{1}{n}\mathbf{1}\mathbf{1}^{\top}  \mX_{T}}\right)\right]\\
&\leq f(\bar\vx_0) - \inf _{\vx \in \BR^d} f(\vx)+\frac{4\eta}{\sqrt{n}}\Norm{\mV_{0}-\frac{1}{n}\mathbf{1}\mathbf{1}^{\top} \mV_{0}} \\
&= \Delta +\frac{4\eta}{\sqrt{n}}\Norm{\mV_{0}-\frac{1}{n}\mathbf{1}\mathbf{1}^{\top}\mV_{0}}.
\end{align*}
For the last term in above equation, we have
\begin{align*}
        &\BE\left[\Norm{\mV_{0}-\frac{1}{n}\boldsymbol{1}\boldsymbol{1}^{\top} \mV_{0}}\right]\\
    \leq &c_1\left(1-c_2\sqrt{1-\beta}\,\right)^{\hat{K}}\BE\left[\Norm{\mG_0-\frac{1}{n}\boldsymbol{1}\boldsymbol{1}^\top\mG_0}\right]\\
    \leq & c_1\left(1-c_2\sqrt{1-\beta}\,\right)^{\hat{K}}\BE\left[\Norm{\mG_0}\right] \\
    \leq &c_1\left(1-c_2\sqrt{1-\beta}\,\right)^{\hat{K}}(\Norm{\nabla f(\mX_0)} + \BE\left[\Norm{\mG_0-\nabla f(\mX_0)})\right] \\
    \leq &c_1\left(1-c_2\sqrt{1-\beta}\,\right)^{\hat{K}}\left(\sqrt{\sum_{i=0}^n\Norm{\nabla f_i(\bar\vx_0)}^2} + \frac{2\sqrt{2n}\sigma}{b^{1-1/p}}\right) \\
    \leq & \frac{\sqrt{n}\Delta}{4\eta},
\end{align*}

where the first inequality is based on Proposition \ref{prop:Chebyshev acceleration}, the  second inequality is based on Lemma \ref{lem:undirected_norm_ineq}, the third inequality is based on the triangle inequality, the fourth inequality is based on Lemma \ref{lem:heavy-tailed} and the last step is based on the setting 
 \begin{align*}
     \hat{K} \geq \frac{1}{1-\beta}\ln \left(1+\frac{8\eta  \left(\sqrt{\sum_{i=1}^n\Norm{\nabla f_i(\bar\vx_{0})}^2}+2\sqrt{2n}\sigma/{b^{1-1/p}}\right)}{\sqrt{n}\Delta}\right) = \tilde{{\Theta}} \left(\frac{1}{1-\beta}\right).
 \end{align*}

taking 
\begin{align*}
    T\geq \left\lceil\frac{924L\Delta}{\epsilon^2}\right\rceil, \quad \eta=\frac{\epsilon}{270L}, \quad \text{and}  \quad b \geq \left\lceil \frac{1}{n} \left(\frac{288\sigma}{\epsilon}\right)^{\frac{p}{p-1}}\right\rceil,
\end{align*}
we obtain
\begin{align*}
    \BE\left[\frac{1}{T}\sum_{t=0}^{T-1}\Norm{\nabla f(\vw_t)}\right]
\leq \frac{\epsilon}{6} + \frac{\epsilon}{6} + \frac{\epsilon}{6} \leq \frac{\epsilon}{2}.
\end{align*}

\end{proof}


\section{The Comparison with Existing Results for Undirected Networks}\label{appendix:compare}

In a recent work, \citet[Theorem 1]{yu2025decentralized} proposed normalized stochastic gradient descent with momentum and gradient tracking (GT-NSGDm) for decentralized nonconvex stochastic optimization over undirected networks, which guarantees  
\begin{align*}
\frac{1}{n T} \sum_{t=0}^{T-1} \sum_{i=1}^{n} \BE\left[\Norm{\nabla f(\vx_{t,i})}\right]=&\mathcal{O}\left(\frac{\sigma}{n^{\frac{p-1}{p}} T^{\frac{p-1}{3 p-2}}}+\frac{1}{T^{\frac{p-1}{3 p-2}}} \sqrt{\frac{L \Delta}{1-\beta}}+\frac{\Norm{\nabla f(\bar{\vx}_{0})}}{T^{\frac{2 p-2}{3 p-2}}}+\sqrt{\frac{(1+ L) \Delta}{(1-\beta) T}}\right. \\
&~~\quad\left.+\sqrt{\frac{n^{\frac{1}{2}} L \Delta}{(1-\beta)^{2} T}}+\frac{\sigma n^{\frac{1}{2}}}{(1-\beta)^{\frac{1}{p}} T^{\frac{p}{3 p-2}}}+\frac{\Norm{\nabla \mF(\mathbf{1}\bx_{0})}}{(1-\beta) n^{\frac{1}{2}} T^{\frac{p}{3 p-2}}}+\ \frac{\sigma n^{\frac{1}{2}}}{(1-\beta)T^{\frac{2 p-1}{3 p-2}}}+\frac{\Delta}{T}\right),
\end{align*}
where $T$ is the number of iterations.
For each iteration, GT-NSGDm takes the communication complexity of $\fO(1)$ and the sample complexity of $\fO(1)$ per agent.
Therefore, for achieving an $\epsilon$-stationary point, GT-NSGDm requires
the overall sample complexity of
\begin{align*}
& \fO(nT)=\mathcal{O}\left(\frac{\sigma ^{\frac{3p-2}{p-1}}}{n^{\frac{2p-2}{p}}\epsilon^{\frac{3p-2}{p-1}}} + \frac{nL^{\frac{3p-2}{2p-2}}\Delta^{\frac{3p-2}{2p-2}}}{(1-\beta)^{\frac{3p-2}{2p-2}}\epsilon^{\frac{3p-2}{p-1}}} + \frac{n\Norm{\nabla f(\bar{\vx}_0)}^{\frac{3p-2}{2p-2}}}{\epsilon^{\frac{3p-2}{2p-2}}}+\frac{n(1+L)\Delta}{(1-\beta)\epsilon^2} \right.\\
& \qquad\qquad\qquad\left.+\frac{n^\frac{3}{2}L\Delta_0}{(1-\beta)^2\epsilon^2} + \frac{\sigma^{\frac{3p-2}{p}}n^{\frac{5p-2}{2p}}}{(1-\beta)^{\frac{3p-2}{p^2}}\epsilon^{\frac{3p-2}{p}}} + \frac{n\Norm{\nabla \mF(\mathbf{1}\bx_{0})}^{\frac{3p-2}{p}}}{(1-\beta)^{\frac{3p-2}{p}}n^{\frac{3p-2}{2p}}\epsilon^{\frac{3p-2}{p}}} + \frac{n\sigma^{\frac{3p-2}{2p-1}}n^{\frac{3p-2}{4p-2}}}{(1-\beta)^{\frac{3p-2}{2p-1}}\epsilon^{\frac{3p-2}{2p-1}}}+\frac{n\Delta}{\epsilon}\right).
\end{align*}
and  the communication complexity of
\begin{align*}
& \fO(T)=\mathcal{O}\left(   \frac{\sigma ^{\frac{3p-2}{p-1}}}{n^{\frac{3p-2}{p}}\epsilon^{\frac{3p-2}{p-1}}} + \frac{L^{\frac{3p-2}{2p-2}}\Delta^{\frac{3p-2}{2p-2}}}{(1-\beta)^{\frac{3p-2}{2p-2}}\epsilon^{\frac{3p-2}{p-1}}} + \frac{\Norm{\nabla f(\bar{\vx}_0)}^{\frac{3p-2}{2p-2}}}{\epsilon^{\frac{3p-2}{2p-2}}}+\frac{(1+L)\Delta}{(1-\beta)\epsilon^2} \right.\\
& \qquad\qquad\qquad\left.+\frac{n^\frac{1}{2}L\Delta_0}{(1-\beta)^2\epsilon^2}+\frac{\sigma^{\frac{3p-2}{p}}n^{\frac{3p-2}{2p}}}{(1-\beta)^{\frac{3p-2}{p^2}}\epsilon^{\frac{3p-2}{p}}}+\frac{\Norm{\nabla \mF(\mathbf{1}\bx_{0})}^{\frac{3p-2}{p}}}{(1-\beta)^{\frac{3p-2}{p}}n^{\frac{3p-2}{2p}}\epsilon^{\frac{3p-2}{p}}}+\frac{\sigma^{\frac{3p-2}{2p-1}}n^{\frac{3p-2}{4p-2}}}{(1-\beta)^{\frac{3p-2}{2p-1}}\epsilon^{\frac{3p-2}{2p-1}}}+\frac{\Delta}{\epsilon}\right).
\end{align*}
Note that the above overall sample complexity is optimal to the accuracy $\epsilon$, while the optimality does not hold for other parameters and the communication.

\section{Experiments on Undirected Networks}\label{appendix:exp-undirected}

Following the setting in Section \ref{sec:experiments}, we also conduct the experiments on training a 6-layer Transformer-XL network, but consider the topologies of undirected networks.
We compare the performance of decentralized normalized stochastic gradient descent (DNSGD) (Algorithm \ref{alg:DNSGD}) \citep{luo2025decentralized}, decentralized stochastic gradient descent  with gradient tracking (D-SGT) \cite{nedic2009distributed,qu2017harnessing}, and gradient tracking normalized stochastic gradient descent with momentum (GTNSGDm) \citep{yu2025decentralized}

We perform our experiments on the networks associated with the undirected ring graph ($n=16$ and $\beta = 0.939$) and the Erd\H{o}s--R\'{e}nyi graph with the connectivity probability of 0.5 ($n=8$ and $\beta=0.839$) to evaluate the sample complexity and the communication complexity.
We set the number of consensus steps as $K \in \{2,3,4\}$ for the undirected ring graph and $K \in \{1,2,3\}$ for the undirected Erd\H{o}s--R\'{e}nyi graph.
For all algorithms, we tune the step size from $\{0.005,0.01,0.02,0.05,0.1,0.2,0.3,0.4,0.5\}$ and fix the minibatch size as 20. 
We additionally run DNSGD with $n\in\{1,2,4,8\}$ to evaluate its speedup with respect to the number of agents.

Figures~\ref{fig:undirected-ring}(a),  \ref{fig:undirected-ring}(b), \ref{fig:undirected-er}(a), and \ref{fig:undirected-er}(b) demonstrate that  DNSGD significantly outperforms the other methods. 
These results validate our theoretical analysis, showing that the normalization step effectively mitigates the impact of heavy-tailed gradient noise.
Moreover, as shown in Figures~\ref{fig:undirected-ring}(a) and \ref{fig:undirected-ring}(b), increasing $K$ from 2 to 3 does improve the performance of DNSGD, whereas further increasing $K$ from 3 to 4 yields no significant gains. 
A similar trend is also observed on the Erd\H{o}s--R\'{e}nyi graph in Figures~\ref{fig:undirected-er}(a) and \ref{fig:undirected-er}(b).
Additionally, Figures~\ref{fig:undirected-ring}(c) and \ref{fig:undirected-er}(c) demonstrate our DNSGD enjoys the linear speedup with respect to the number of agents $n$, matching our theoretical analysis in Section \ref{sec:undirected}.

\newpage

\begin{figure}[t]
  \centering
  \begin{tabular}{c@{\hspace{2mm}}c@{\hspace{2mm}}c}
    \includegraphics[width=0.31\columnwidth]{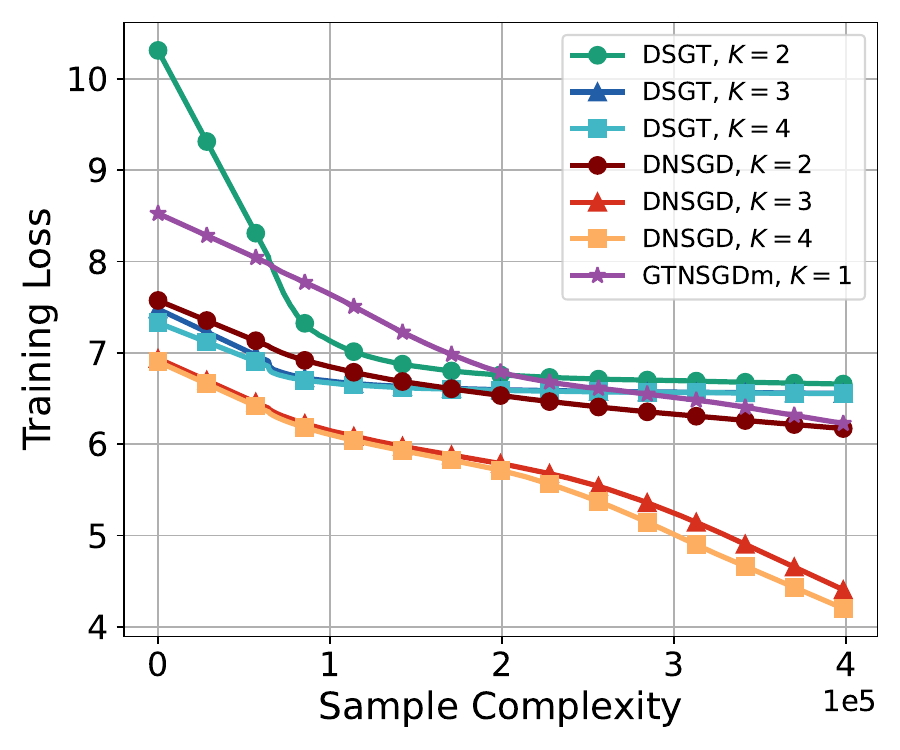} &
    \includegraphics[width=0.31\columnwidth]{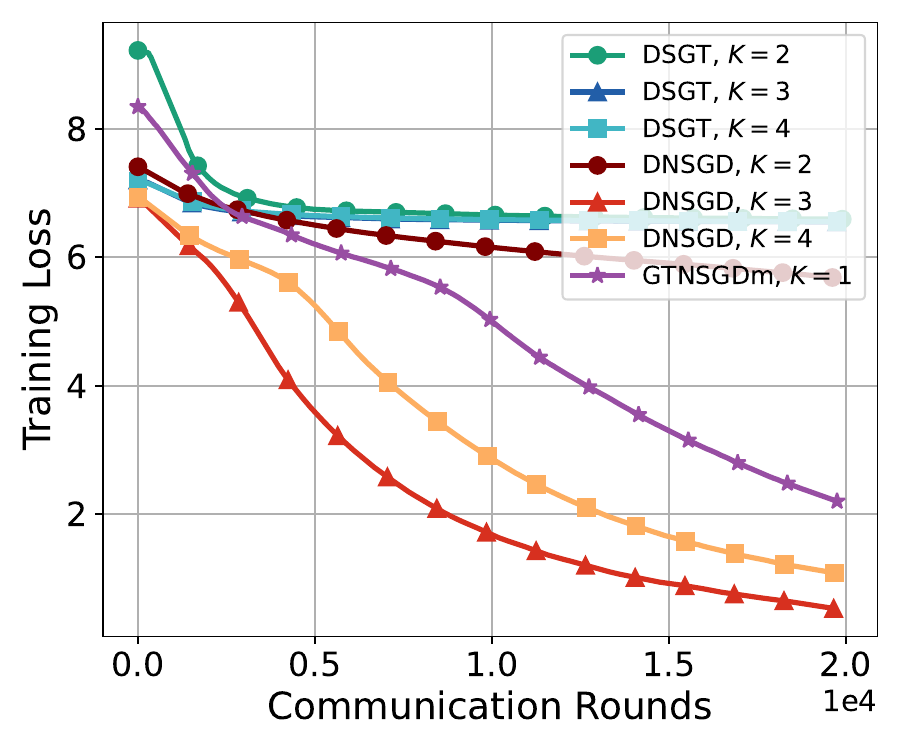} &
    \includegraphics[width=0.31\columnwidth]{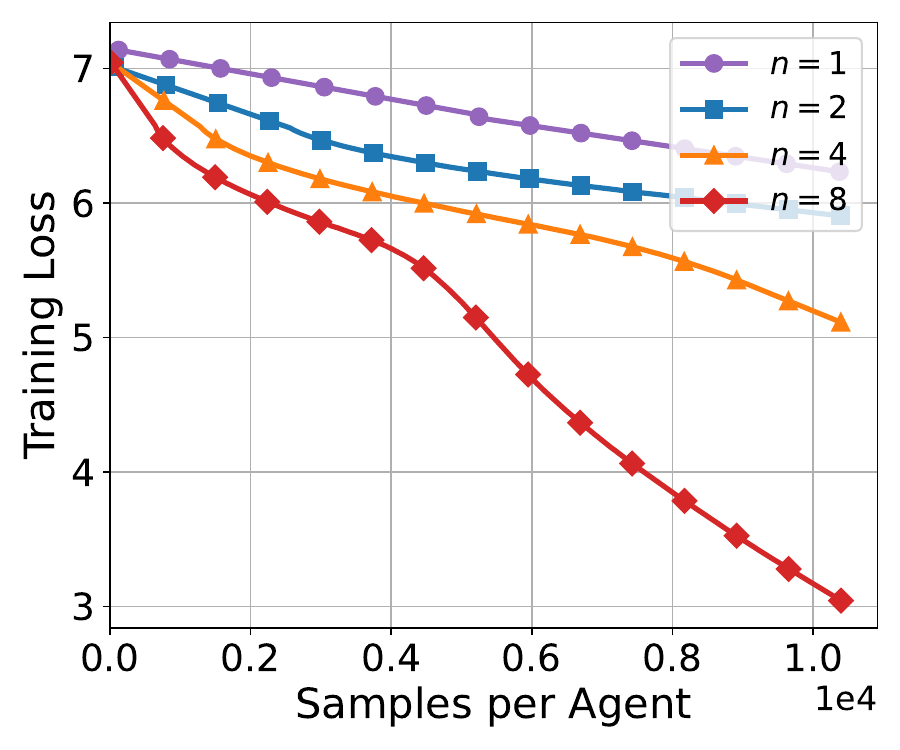} \\
    \small (a) \#sample vs.\ training loss &
    \small (b) \#communication vs.\ training loss &
    \small (c) speedup with respect to $n$
  \end{tabular}
  \caption{The comparison of DNSGD, DSGT, and GTNSGDm on the undirected ring network, 
  where subfigures (a) and (b) set $n=16$ and $K\in\{2,3,4\}$, 
  and subfigure (c) sets $n\in\{1,2,4,8\}$ and $K=2$.}
  \label{fig:undirected-ring}
  \vskip -0.1cm
\end{figure}

\begin{figure}[t]
  \centering
  \begin{tabular}{c@{\hspace{2mm}}c@{\hspace{2mm}}c}
    \includegraphics[width=0.31\columnwidth]{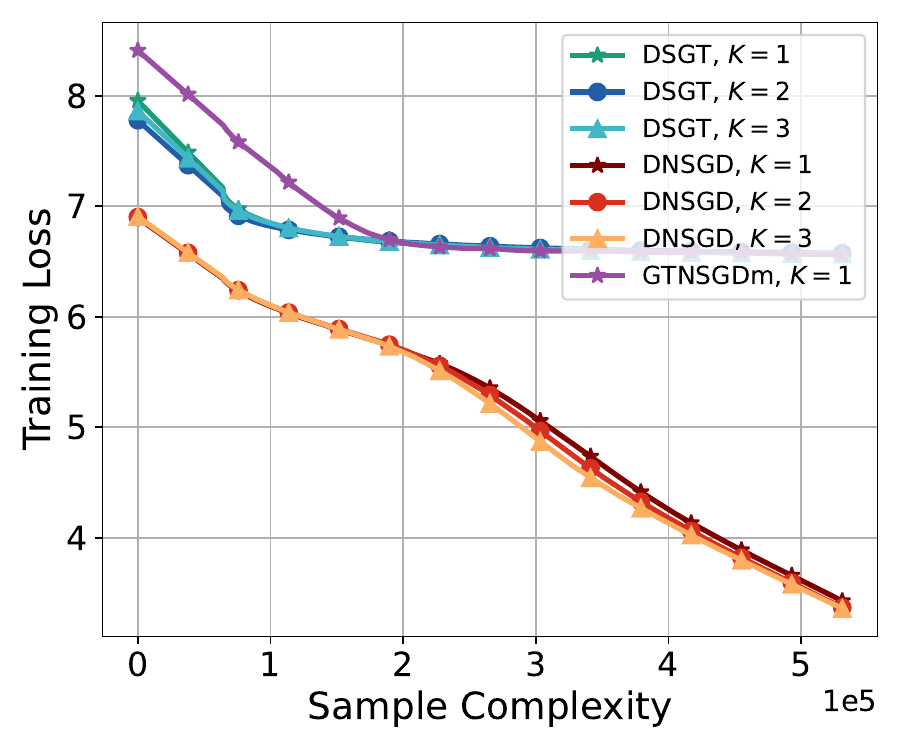} &
    \includegraphics[width=0.31\columnwidth]{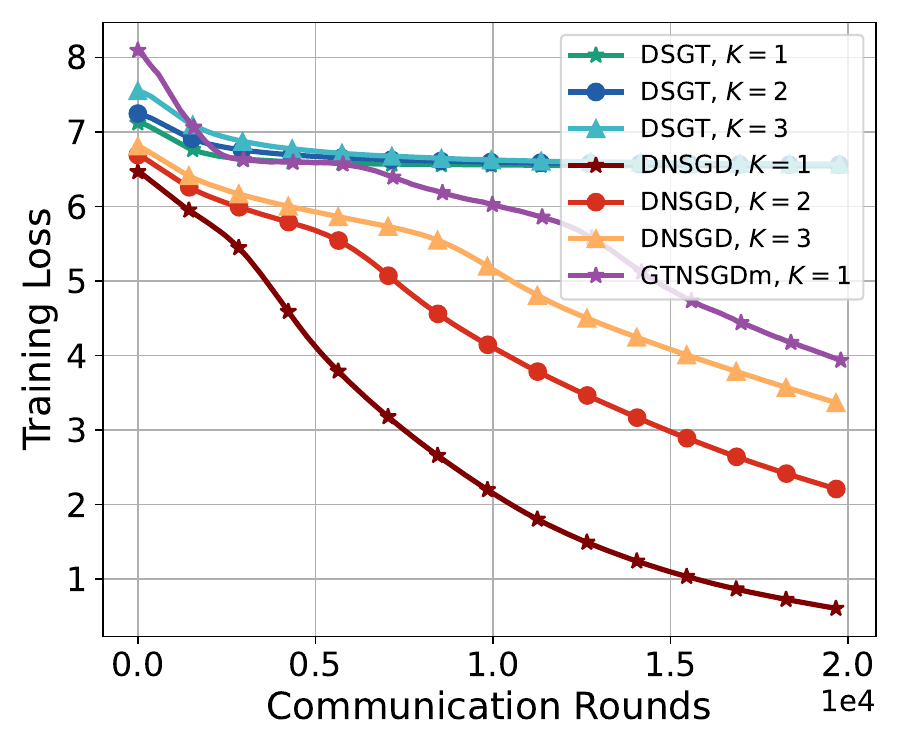} &
    \includegraphics[width=0.31\columnwidth]{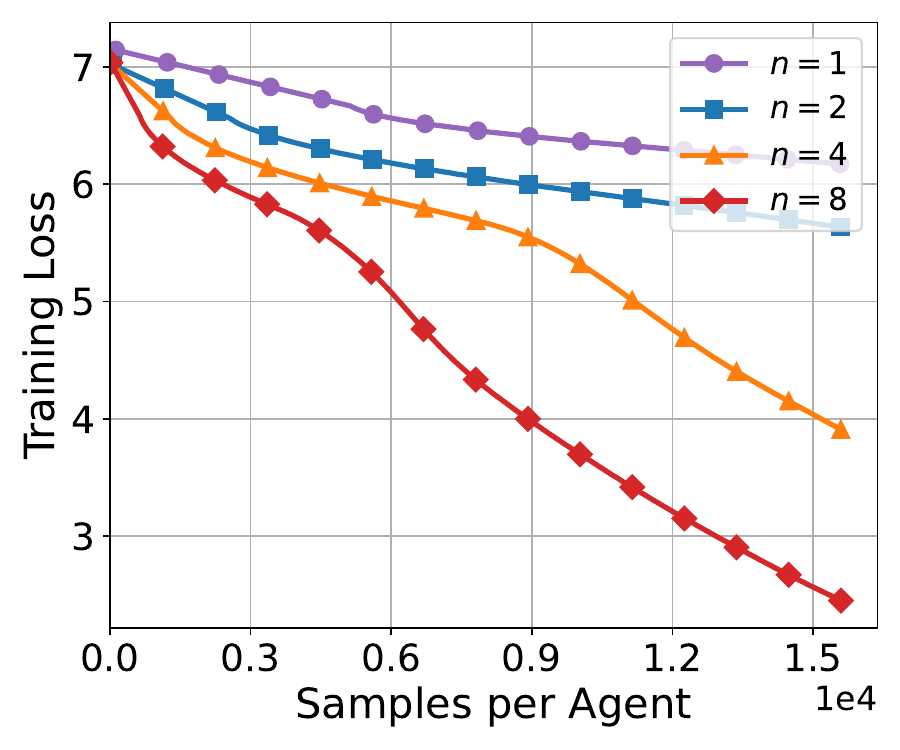} \\
    \small (a) \#sample vs.\ training loss &
    \small (b) \#communication vs.\ training loss &
    \small (c) speedup with respect to $n$
  \end{tabular}
  \caption{The comparison of DNSGDm, DSGT, and GTNSGDm on the undirected Erd\H{o}s--R\'{e}nyi network, 
  where subfigures (a) and (b) set $n=8$ and $K\in\{1,2,3\}$, 
  and subfigure (c) sets $n\in\{1,2,4,8\}$ and $K=1$.}
  \label{fig:undirected-er}
  \vskip -0.1cm
\end{figure}

\end{document}

%% file: math.tex
\usepackage{amsmath}\allowdisplaybreaks
\usepackage{amsfonts,bm}
\usepackage{amssymb}

\def\eqref#1{equation~(\ref{#1})}

\def\1{\bf{1}}

\newcommand{\Norm}[1]{\left\| #1 \right\|}

\def\vzero{{\bf{0}}}
\def\vone{{\bf{1}}}

\def\va{{\bf{a}}}
\def\vb{{\bf{b}}}

\def\vg{{\bf{g}}}

\def\vv{{\bf{v}}}
\def\vw{{\bf{w}}}
\def\vx{{\bf{x}}}
\def\vy{{\bf{y}}}

\def\fD{{\mathcal{D}}}
\def\fE{{\mathcal{E}}}

\def\fG{{\mathcal{G}}}

\def\fO{{\mathcal{O}}}

\def\fV{{\mathcal{V}}}

\def\BE{{\mathbb{E}}}

\def\BN{{\mathbb{N}}}

\def\BR{{\mathbb{R}}}

\def\mA {{\bf A}}

\def\mD {{\bf D}}

\def\mF {{\bf F}}
\def\mG {{\bf G}}

\def\mI {{\bf I}}

\def\mU {{\bf U}}
\def\mV {{\bf V}}

\def\mX {{\bf X}}
\def\mY {{\bf Y}}
\def\mZ {{\bf Z}}

\usepackage{amsthm}
\theoremstyle{plain}
\newtheorem{thm}{Theorem}

\newtheorem{lem}{Lemma}
\newtheorem{asm}{Assumption}
\newtheorem{remark}{Remark}

\newtheorem{prop}{Proposition}

\makeatletter
\def\Ddots{\mathinner{\mkern1mu\raise\p@
\vbox{\kern7\p@\hbox{.}}\mkern2mu
\raise4\p@\hbox{.}\mkern2mu\raise7\p@\hbox{.}\mkern1mu}}
\makeatother

\makeatletter
\newcommand*{\rom}[1]{\expandafter\@slowromancap\romannumeral #1@}
\makeatother

\def\bx{{\bar\vx}}

\def\vxi{{\bm\xi}}

\def\AG{{\rm AccGossip}}
\def\vpi{{\bm\pi}}
\def\vdelta{{\bm\delta}}
\def\mPi{{\bf\Pi}}

\newcommand{\pushsum}{{\sc Push-Sum}\xspace}
\newcommand{\pushpull}{{\sc Push-Pull}\xspace}
\newcommand{\pulldiag}{{\sc Pull-Diag}\xspace}